\pdfoutput=1
\RequirePackage{ifpdf}
\ifpdf 
\documentclass[pdftex]{sigma}
\else
\documentclass{sigma}
\fi

{\theoremstyle{definition}
\newtheorem{dfn}{Definition}[section]
\newtheorem{rem}[dfn]{Remark}
\newtheorem{exa}[dfn]{Example}

}

\newtheorem{prop}[dfn]{Proposition}
\newtheorem{thm}[dfn]{Theorem}
\newtheorem{lem}[dfn]{Lemma}
\newtheorem{cor}[dfn]{Corollary}

\usepackage{tikz-cd,upgreek,dsfont,euscript}
\usetikzlibrary{cd,calc}

\def\ang<#1>{\langle #1 \rangle}
\def\bigang<#1>{\left\langle #1 \right\rangle}

\makeatletter
\newcommand*{\defeq}{\mathrel{\rlap{%
 \raisebox{0.3ex}{$\m@th\cdot$}}%
 \raisebox{-0.3ex}{$\m@th\cdot$}}%
 =}
\makeatother

\tikzset{
 labl/.style={anchor=south, rotate=90, inner sep=.5mm}
}

\numberwithin{equation}{section}

\newcommand{\Ch}{\operatorname{Ch}}
\newcommand{\Chb}{\operatorname{Ch^{b}}}

\newcommand{\Fact}{\operatorname{Fact}}
\newcommand{\dgFact}{\operatorname{\sf{Fact}}}

\newcommand{\Kadd}{\operatorname{Kadd}}

\newcommand{\Dmod}{\operatorname{Dmod}}

\newcommand{\DcoMod}{\operatorname{D^{co}\hspace{-0.1mm}Mod}}

\newcommand{\DMF}{\operatorname{DMF}}
\newcommand{\KMF}{\operatorname{KMF}}

\newcommand{\Dcoh}{\operatorname{Dcoh}}

\newcommand{\DcoQcoh}{\operatorname{D^{co}\hspace{-0.1mm}Qcoh}}

\newcommand{\Sing}{\operatorname{Sing}}
\newcommand{\Spec}{\operatorname{Spec}}

\newcommand{\Hom}{\operatorname{Hom}}
\newcommand{\cHom}{\operatorname{{\mathcal H}om}}
\newcommand{\Ext}{\operatorname{Ext}}

\newcommand{\Tot}{\operatorname{Tot}}
\newcommand{\RHom}{\operatorname{{\mathbb R}Hom}}
\newcommand{\End}{\operatorname{End}}
\newcommand{\cEnd}{\operatorname{{\mathcal E}nd}}

\newcommand{\coh}{\operatorname{coh}}
\newcommand{\Qcoh}{\operatorname{Qcoh}}

\newcommand{\Mod}{\operatorname{Mod}}
\newcommand{\Proj}{\operatorname{Proj}}

\newcommand{\fmod}{\operatorname{mod}}

\newcommand{\add}{\operatorname{add}}
\newcommand{\Add}{\operatorname{Add}}
\newcommand{\vect}{\operatorname{vect}}
\newcommand{\lfr}{\operatorname{lfr}}
\newcommand{\proj}{\operatorname{proj}}
\newcommand{\Db}{\operatorname{D^b}}
\newcommand{\Dco}{\operatorname{D^{co}}}
\newcommand{\D}{\operatorname{D}}

\newcommand{\Aco}{\operatorname{A^{co}}}
\newcommand{\A}{\operatorname{A}}
\newcommand{\K}{\operatorname{K}}

\newcommand{\Kb}{\operatorname{K^b}}

\newcommand{\id}{\operatorname{id}}
\newcommand{\Pf}{\operatorname{Pf}}
\newcommand{\diag}{\operatorname{diag}}

\newcommand{\Ker}{\operatorname{Ker}}

\newcommand{\Sp}{\operatorname{Sp}}
\newcommand{\GL}{\operatorname{GL}}
\newcommand{\GSp}{\operatorname{GSp}}
\newcommand{\Gr}{\operatorname{Gr}}

\newcommand{\Res}{\operatorname{\sf Res}}
\newcommand{\Ind}{\operatorname{\sf Ind}}

\newcommand{\aff}{\operatorname{\sf aff}}
\newcommand{\rk}{\operatorname{\sf rk}}

\newcommand{\stab}{\operatorname{ss}}
\newcommand{\DB}{\operatorname{DB}}
\newcommand{\rep}{\operatorname{rep}}

\newcommand{\Sch}{\operatorname{Sch\!}}
\newcommand{\fppf}{\operatorname{fppf}}
\newcommand{\Rings}{\operatorname{Rings}}
\newcommand{\Sets}{\operatorname{Sets}}
\newcommand{\op}{\operatorname{op}}
\newcommand{\can}{\operatorname{can}}
\newcommand{\Sh}{\operatorname{Sh}}
\newcommand{\et}{\text{\rm \'et}}

\newcommand{\cA}{\mathcal{A}}
\newcommand{\cB}{\mathcal{B}}
\newcommand{\cC}{\mathcal{C}}

\newcommand{\cE}{\mathcal{E}}
\newcommand{\cF}{\mathcal{F}}
\newcommand{\cG}{\mathcal{G}}

\newcommand{\cI}{\mathcal{I}}
\newcommand{\cJ}{\mathcal{J}}
\newcommand{\cK}{\mathcal{K}}

\newcommand{\cM}{\mathcal{M}}
\newcommand{\cN}{\mathcal{N}}
\newcommand{\cO}{\mathcal{O}}

\newcommand{\cT}{\mathcal{T}}

\newcommand{\cX}{\mathcal{X}}
\newcommand{\cY}{\mathcal{Y}}

\newcommand{\bA}{\mathbb{A}}
\newcommand{\bC}{\mathbb{C}}
\newcommand{\bG}{\mathbb{G}}

\newcommand{\bL}{\mathbb{L}}

\newcommand{\bP}{\mathbb{P}}

\newcommand{\bR}{\mathbb{R}}

\newcommand{\bZ}{\mathbb{Z}}

\newcommand{\scrA}{\EuScript{A}}
\newcommand{\scrB}{\EuScript{B}}
\newcommand{\scrC}{\EuScript{C}}
\newcommand{\scrD}{\EuScript{D}}
\newcommand{\scrE}{\EuScript{E}}
\newcommand{\scrF}{\EuScript{F}}

\newcommand{\scrI}{\EuScript{I}}

\newcommand{\scrP}{\EuScript{P}}
\newcommand{\scrQ}{\EuScript{Q}}

\newcommand{\scrT}{\EuScript{T}}

\newcommand{\scrV}{\EuScript{V}}

\newcommand{\scrY}{\EuScript{Y}}

\newcommand{\F}{\operatorname{\sf{F}}}
\newcommand{\G}{\operatorname{\sf{G}}}

\newcommand{\simto}{\,\,\,\,\raisebox{4pt}{$\sim$}{\kern -1.1em \longrightarrow}\,}
\newcommand{\hookto}{\hookrightarrow}
\newcommand{\surjto}{\twoheadrightarrow}
\newcommand{\timesH}{\times^{\hspace{-0.2mm}H\hspace{0.2mm}}\!\!}
\newcommand{\gwedge}{\raisebox{1pt}[0pt][0pt]{$\textstyle \bigwedge^{\raisebox{-1pt}[0pt][0pt]{\tiny\!$2$\,}}$}}

\newcommand{\rest}{\!|_L}

\begin{document}

\newcommand{\arXivNumber}{2009.12785}

\renewcommand{\thefootnote}{}

\renewcommand{\PaperNumber}{055}

\FirstPageHeading

\ShortArticleName{Equivariant Tilting Modules, Pfaffian Varieties and Noncommutative Matrix Factorizations}

\ArticleName{Equivariant Tilting Modules, Pfaffian Varieties\\ and Noncommutative Matrix Factorizations\footnote{This paper is a~contribution to the Special Issue on Primitive Forms and Related Topics in honor of~Kyoji Saito for his 77th birthday. The full collection is available at \href{https://www.emis.de/journals/SIGMA/Saito.html}{https://www.emis.de/journals/SIGMA/Saito.html}}}

\Author{Yuki HIRANO}

\AuthorNameForHeading{Y.~Hirano}

\Address{Department of Mathematics, Kyoto University,\\ Kitashirakawa-Oiwake-cho, Sakyo-ku, Kyoto, 606-8502, Japan}
\Email{\href{mailto:y.hirano@math.kyoto-u.ac.jp}{y.hirano@math.kyoto-u.ac.jp}}

\ArticleDates{Received September 29, 2020, in final form May 28, 2021; Published online June 02, 2021}

\Abstract{We show that equivariant tilting modules over equivariant algebras induce equivalences of derived factorization categories. As an application, we show that the derived cate\-gory of a noncommutative resolution of a linear section of a Pfaffian variety is equivalent to the derived factorization category of a noncommutative gauged Landau--Ginzburg model $(\Lambda,\chi, w)^{\mathbb{G}_m}$, where $\Lambda$ is a noncommutative resolution of the quotient singularity $W/\operatorname{GSp}(Q)$ arising from a certain representation $W$ of the symplectic similitude group $\operatorname{GSp}(Q)$ of a~symplectic vector space~$Q$.}


\Classification{14F08; 18G80; 16E35}

\begin{flushright}
\begin{minipage}{60mm}
\it Dedicated to Professor Kyoji~Saito \\ on the occasion of his~77th~birthday
\end{minipage}
\end{flushright}

\renewcommand{\thefootnote}{\arabic{footnote}}
\setcounter{footnote}{0}

\section{Introduction}

\subsection{Backgrounds}
Derived categories of coherent sheaves on varieties are one of the important invariants in algebraic geometry, and they have interesting links to other fields of mathematics. For example, if a smooth variety $X$ admits a tilting bundle $\cT$, the derived category $\Db(\coh X)$ of coherent sheaves on $X$ is equivalent to the derived category $\Db(\fmod\Lambda)$ of finitely generated right modules over the endomorphism ring $\Lambda\defeq\End_X(\cT)$ of $\cT$.
Once we have such an equivalence, we can study the derived category of coherent sheaves by the representation theory of noncommutative algebras. However, if $X$ is a smooth projective Calabi--Yau variety, $X$ can never admit a tilting bundle.

Recently, Okonek--Teleman proved that the derived category of a regular zero section in a~certain smooth variety is equivalent to the derived factorization category of a noncommutative gauged Landau--Ginzbubrg (LG) model~\cite{ot}. In particular, they proved that the derived category $\Db(\coh Z)$ of a Calabi--Yau complete intersection $Z\subset \bP^n$ is equivalent to the derived factorization category $\Dmod_{\bG_m}(\Lambda,\chi,w)$ of a noncommutative gauged LG model $(\Lambda,\chi,w)^{\bG_m}$, where $\Lambda$ is a noncommutative crepant resolution of some quotient singularity. This result gives a new approach to the study of the derived categories of Calabi--Yau complete intersections. For example, it is interesting to interpret autoequivalences of $\Db(\coh Z)$ with autoequivalences of $\Dmod_{\bG_m}(\Lambda,\chi,w)$ induced by $\bG_m$-equivariant tilting modules/complexes over $\Lambda$, and it is expected that this interpretation enables us to study the fundamental group action on $\Db(\coh Z)$, constructed in~\cite{hl-s} using variations of GIT quotients, by equivariant tilting theory over $\Lambda$.

In this paper, for any reductive affine algebraic group $G$,
we generalize Okonek--Teleman's result to a $G$-equivariant setting. More precisely, we prove that $G$-equivariant tilting modules over $G$-equivariant algebras induce equivalences of derived factorization categories of noncommutative gauged LG models. Moreover, combining Rennemo--Segal's results in~\cite{rs} with our result, we prove that the derived category of a noncommutative resolution of a generic linear section of a Pfaffian variety is equivalent to the derived factorization category of a noncommutative gauged LG model $(\Lambda, \chi, w)^{\mathbb{G}_m}$, where $\Lambda$ is a non-commutative resolution of the quotient singularity $W/\GSp(Q)$ arising from a certain representation $W$ of the symplectic similitude group $\GSp(Q)$ of a symplectic vector space $Q$.

\subsection{Equivariant algebras and noncommutative LG models}

Let $X$ be a scheme, and $G$ an algebraic group acting on $X$. Denote by $\sigma\colon G\times X\to X$ the morphism defining the $G$-action on $X$, and write $\pi\colon G\times X\to X$ for the natural projection. Let~$\uprho\colon \cO_X\to\cA$ be a (not necessarily commutative) $\cO_X$-algebra that is a coherent $\cO_X$-module. A~{\it $G$-equivariant structure} on $\cA$ is an isomorphism $\uptheta^{\cA}\colon\pi^*\cA\simto\sigma^*\cA$ of sheaves of $\cO_{G\times X}$-algebras such that the pair $\big(\cA,\uptheta^{\cA}\big)$ defines the $G$-equivariant coherent $\cO_X$-module, and we call the pair $\big(\cA,\uptheta^{\cA}\big)$ a {\it $G$-equivariant coherent $X$-algebra}. Similarly, we define a {\it $G$-equivariant coherent $\cA$-module} to be the pair $\big(\cM,\uptheta^{\cM}\big)$ of a right $\cA$-module $\cM$ and an isomorphism $\uptheta^{\cM}\colon\pi^*\cM\simto\sigma^*\cM$ of $\pi^*\cA$-modules such that the pair $\big(\cM,\uptheta^{\cM}\big)$ defines a $G$-equivariant coherent $\cO_X$-module. We~write $\coh_G\cA$ for the category of $G$-equivariant coherent $\cA$-modules. Then we have an equivalence
\begin{gather*}
\coh_G\cA\cong\coh[\cA/G]
\end{gather*}
of categories, where $[\cA/G]$ is the associated coherent $[X/G]$-algebra (see Appendix~\ref{appendix}).

Let $\chi\colon G\to \bG_m$ be a character of $G$, and $w\colon X\to \bA^1$ a $\chi$-semi-invariant regular function on $X$, i.e., $W(g\cdot x)=\chi(g)W(x)$ for any $(g,x)\in G\times X$. We call the data $(\cA,\chi,w)^G$ a {\it gauged Landau--Ginzburg $(LG)$ model}, and it is said to be {\it noncommutative} if the algebra $\cA$ is not commutative. We define a {\it factorization} of $(\cA,\chi,w)^G$ to be a sequence
\begin{gather*}
\cM_1\xrightarrow{\varphi_1}\cM_0\xrightarrow{\varphi_0}\cO_X(\chi)\otimes_{\cO_X}\cM_1
\end{gather*}
consisting of $G$-equivariant coherent $\cA$-modules $\cM_i$ and $G$-equivariant $\cA$-linear maps $\varphi_i$ such that the compositions $\varphi_0\circ \varphi_1$ and $\varphi_1(\chi)\circ\varphi_0$ are the multiplications by $w$.
To a gauged LG model $(\cA,\chi,w)^G$ we associate the {\it derived factorization category}
\begin{gather*}
\Dcoh_G(\cA,\chi,w)
\end{gather*}
whose objects are factorizations of $(\cA,\chi,w)^G$. If $X=\Spec R$, the $G$-equivariant $X$-algebra $\cA$ corresponds to the $G$-equivariant $R$-algebra $A\defeq\Gamma(X,\cA)$, and we write $\Dmod_G(A,\chi,w)$ for the corresponding derived factorization category.

\subsection{Results}
 Let $X$ be a variety, $G$ an affine algebraic group acting on $X$, and $H$ a closed normal subgroup of $G$.
 \begin{dfn}
 A $G$-equivariant coherent $\cA$-module $\cT\in \coh_G\cA$ is called a {\it $(G,H)$-tilting} (resp.\ {\it partial $(G,H)$-tilting}) if the restriction $\cT_H\in \coh_H\cA$ satisfies the following conditions (1), (2) and (3) (resp.\ (1) and (2)):
 \begin{itemize}\itemsep=0pt
\item[$1.$] We have $\Ext_{\coh_H\cA}^i(\cT_H,\cT_H)=0$ for all $i>0$.

\item[$2.$] The object $\cT_H$ is {\it compact} in $\D\big(\Qcoh_H\cA\big)$.

\item[$3.$] The object $\cT_H$ is a {\it generator} of $\D\big(\Qcoh_H\cA\big)$, i.e., for any complex $A^{\bullet}$ in $\Qcoh_H\cA$, $\RHom(T,A^{\bullet})\cong0$ in $\D(\Mod\End(\cT_H))$ implies $A^{\bullet}\cong0$ in $\D\big(\Qcoh_H\cA\big)$.
\end{itemize}
 \end{dfn}

Let $\Spec R$ be an affine variety with a $G$-action such that the induced $H$-action is trivial so that $\Spec R$ has the induced $G/H$-action.
Let $f\colon X\to \Spec R$ be a $G$-equivariant morphism such that the associated morphism $\overline{f}\colon[X/H]\to \Spec R$ is proper. Let $\cT\in \coh_G\cA$ be a $G$-equivariant coherent $\cA$-module, and write $\Lambda\defeq\End_{\coh\cA}(\cT)$ for the endomorphism ring of the underlying $\cA$-module $\cT$. Then $\Lambda$ has a natural $H$-action, and the $H$-invariant ring $\Lambda^H$ is a~$G/H$-equivariant coherent $R$-algebra.
Let $\chi\colon G/H\to \bG_m$ be a character of $G/H$, and define the character $\widehat{\chi}\colon G\to \bG_m$ to be the composition of $\chi$ and the natural projection $G\to G/H$. Let $w_R\colon \Spec R\to \bA^1$ be a $\chi$-semi-invariant regular function on $\Spec R$, and write $w\defeq f^*w_R$. Then we have two gauged LG models
\begin{gather*}
\big(\cA,\widehat{\chi},w\big)^G \qquad\text{and}\qquad \big(\Lambda^H,\chi,w_R\big)^{G/H}.
\end{gather*}
The following is a generalization of Okonek--Teleman's result~\cite[Theorem~1.11]{ot}.
\begin{thm}[Theorem~\ref{tilt and fact}]\label{main intro}
Assume that $G/H$ is reductive and $\Lambda^H$ is of finite global dimension.
\begin{itemize}\itemsep=0pt
\item[$1.$] If $\cT$ is partial $(G,H)$-tilting, then we have a fully faithful functor
\begin{gather*}
\Dmod_{G/H}\big(\Lambda^H,\chi,w_R\big)\hookto\Dcoh_{G}\big(\cA,\widehat{\chi},w\big).
\end{gather*}
\item[$2.$] If $\cT$ is $(G,H)$-tilting and every quasi-coherent $\cA$-module has a finite injective resolution, we have an equivalecne
\begin{gather*}
\Dcoh_{G}\big(\cA,\widehat{\chi},w\big)\simto\Dmod_{G/H}\big(\Lambda^H,\chi,w_R\big).
\end{gather*}
\end{itemize}
\end{thm}

Let $V$ be a vector space of dimension $v$. For an integer $q$ with $0\leq 2q\leq v$, we have a Pfaffian variety
\begin{gather*}
\Pf_q\defeq\big\{x\in\gwedge V^* \mid \rk(x)\leq 2q\big\}\subseteq \bP\big(\gwedge V^*\big),
\end{gather*}
and we consider its linear section
\begin{gather*}
\Pf_q\rest\defeq\Pf_q\cap\,\,\bP(L),
\end{gather*}
where $L\subset \gwedge V^*$ is a subspace of $V$ such that $\Pf_q\rest\neq\varnothing$. If we choose a generic $L$, there is a noncommutative resolution $\cB_L$ of $\Pf_q\rest$~\cite{rsv,svdb}, and if $\Pf_q\rest$ is smooth $\Db(\coh\cB_L)$ is equivalent to $\Db(\coh\Pf_q\rest)$. Combining Theorem~\ref{main intro} and the results in~\cite{rs}, we have the following.
\begin{cor}[Corollary~\ref{main cor}]
There is a noncommutative gauged LG model $(\Lambda,\chi,w)^{\bG_m}$ such that we have an equivalence
\begin{gather*}
\Db(\coh \cB_L)\cong \Dmod_{\bG_m}(\Lambda,\chi,w)
\end{gather*}
and that $\Lambda$ is a noncommutative resolution of the affine quotient $W/\GSp(Q)$ of a certain representation $W$ of the symplectic similitude group $\GSp(Q)$ of a symplectic vector space $Q$.
\end{cor}

\subsection{Notation and convention}
\begin{itemize}\itemsep=0pt

\item Unless stated otherwise, categories and stacks we consider are over an algebraically closed field $k$ of characteristic zero.

\item For a character $\chi\colon G\to \bG_m$ of an algebraic group $G$, we denote by $\cO_X(\chi)$, or~sim\-ply~$\cO(\chi)$, the $G$-equivariant invertible sheaf on a $G$-scheme $X$ associated to $\chi$.

\item For an exact category $\scrE$, we denote by $\Ch(\scrE)$ (resp.\ $\Chb(\scrE)$) the category of cochain complexes (resp.\ bounded cochain complexes) in $\scrE$, and we write $\D(\scrE)$ (resp.\ $\Db(\scrE)$) for~the derived category (resp.\ the bounded derived category) of $\scrE$.
\end{itemize}

\section{Preliminaries}
In this section, we recall definitions and basic properties about derived factorization categories, equivariant quasi-coherent sheaves and tilting objects.

\subsection{Derived factorization categories}
We recall the basics of derived factorization categories mainly to fix notation. See, for example,~\cite{bdfik,H1,posi} for more details.
Throughout this section, $\scrA$ is an abelian category with small coproducts such that the small coproducts of families of short exact sequences are exact. In~what follows, we fix a triple
\begin{equation}\label{triple}
(\scrE,\Phi,w)
\end{equation}
consisting of an exact subcategory $\scrE\subseteq\scrA$ of $\scrA$, an exact autoequivalence $\Phi\colon \scrA\to \scrA$ pre\-ser\-ving~$\scrE$ and a functor morphism $w\colon \id\to\Phi$ that is compatible with $\Phi$, i.e., the equality $w({\Phi(E)})=\Phi(w(E))$ holds for every object $E\in\scrE$. The functor morphism $w\colon \id\to\Phi$ is called a {\it potential} of $\scrE$.

\begin{dfn} A {\it factorization} of $(\scrE,\Phi,w)$ is a sequence
\begin{gather*}
E=\Bigl(E_1\xrightarrow{\varphi^E_1}E_0\xrightarrow{\varphi^E_0}\Phi(E_1)\Bigr)
\end{gather*}
in $\mathcal{E}$ such that $\varphi^E_0\circ\varphi^E_1=w(E_1)$ and $\Phi\big(\varphi^E_1\big)\circ\varphi^E_0=w(E_0)$. Objects $E_1$ and $E_0$ in the above sequence are called the {\it components} of $E$. \end{dfn}

\begin{dfn}
Let $(\scrE,\Phi,w)$ be the above triple.
\begin{itemize}\itemsep=0pt
\item[$1.$] For two factorizations $E,F$ of $(\scrE,\Phi,w)$, we define the complex $\Hom(E,F)^{\bullet}$ with differential $d^{\bullet}_{(E,F)}\colon\Hom(E,F)^{\bullet}\to\Hom(E,F)^{\bullet+1}$ by
\begin{gather*}
{\rm Hom}(E,F)^{2n}\defeq{\rm Hom}_{\scrE}(E_1,\Phi^n(F_1))\oplus{\rm Hom}_{\scrE}(E_0,\Phi^n(F_0)),
\\
{\rm Hom}(E,F)^{2n+1}\defeq{\rm Hom}_{\scrE}(E_1,\Phi^{n}(F_0))\oplus{\rm Hom}_{\scrE}(E_0,\Phi^{n+1}(F_1)),
\end{gather*}
where $E_i$ and $F_i$ are the components of $E$ and $F$ respectively, and
\begin{gather*}
d^{\bullet}_{(E,F)}(f):=\varphi^F\circ f-(-1)^{{\rm deg}(f)}f\circ \varphi^E
\qquad\text{if}\quad f\in{\rm Hom}(E,F)^{{\rm deg}(f)}.
\end{gather*}
This defines the dg category
\begin{gather*}
\dgFact(\scrE,\Phi,w)
\end{gather*}
of factorizations of $(\scrE,\Phi,w)$.

\item[$2.$] The dg category $\dgFact(\scrE,\Phi,w)$ defines the additive categories
\begin{gather*}
\Fact(\scrE,\Phi,w)\defeq Z^0(\dgFact(\scrE,\Phi,w)),
\\
\K(\scrE,\Phi,w)\defeq H^0(\dgFact(\scrE,\Phi,w)).
\end{gather*}
We call $\K(\scrE,\Phi,w)$ the {\it homotopy category} of $(\scrE,\Phi,w)$.
\end{itemize}
\end{dfn}

The category $\Fact(\scrE,\Phi,w)$ is an exact category, and the homotopy category $\K(\scrE,\Phi,w)$ is a~triangulated category (see~\cite[Propositions~3.5 and~3.9]{H1}). To define the derived factorization categories, we define the totalizations of complexes of factorizations: For an object $E\in \Fact(\scrE,\Phi,w)$, let us set
\begin{gather*}
{\rm Com}(E)^{2i}\defeq\Phi^i(E_0), \qquad {\rm Com}(E)^{2i-1}\defeq\Phi^i(E_1),
\\
d_E^{2i}\defeq\Phi^i\big(\varphi_0^E\big),\qquad d_E^{2i-1}\defeq\Phi^i\big(\varphi_1^E\big).
\end{gather*}
Then a periodic (up to twists by $\Phi$) infinite sequence
\begin{gather*}
{\rm Com}(E)\defeq\big({\rm Com}(E)^{\bullet}, d_E^{\bullet}\big)
\end{gather*}
 in $\scrE$ satisfies $d_E^{i+1}\circ d_E^i=w({\rm Com}(E)^i)$ for all $i\in\bZ$.

\begin{dfn}
Let $E^{\bullet}=\big(\cdots\rightarrow E^i\xrightarrow{\delta^i}E^{i+1}\rightarrow\cdots\big)$ be a bounded complex of $\Fact(\scrE,\Phi,w)$. We define the {\it totalization} $\Tot(E^{\bullet})\in\Fact(\scrE,\Phi,w)$ of $E^{\bullet}$ by
\begin{gather*}
{\rm Tot}(E^{\bullet})\defeq\big(T_1\xrightarrow{t_1}T_0\xrightarrow{t_0}\Phi(T_1)\big),
\end{gather*}
where
\begin{gather*}
T_l:=\bigoplus_{i+j=-l}{\rm Com}(E^i)^j,
\\
t_l|_{{\rm Com}(E^i)^j}\defeq{\rm Com}(\delta^i)^j+(-1)^id_{E^i}^j.
\end{gather*}
Taking totalizations defines an exact functor
\begin{gather*}
\Tot\colon\ \Chb(\Fact(\scrE,\Phi,w))\rightarrow \Fact(\scrE,\Phi,w).
\end{gather*}
\end{dfn}

\begin{dfn}
Let $\A(\scrE,\Phi,w)$ be the smallest thick subcategory of $\K(\scrE,\Phi,w)$ that contains tota\-lizations of all short exact sequences in $\Fact(\scrE,\Phi,w)$. Then we define the {\it derived factorization category} $\D(\scrE,\Phi,w)$ of the triple $(\scrE,\Phi,w)$ by the Verdier quotient
\begin{gather*}
\D(\scrE,\Phi,w)\defeq\K(\scrE,\Phi,w)/\A(\scrE,\Phi,w).
\end{gather*}
If $\scrE$ is closed under coproducts in $\scrA$, we denote by
$\Aco(\scrE,\Phi,w)$ the smallest thick subcategory of $\K(\scrE,\Phi,w)$ that contains totalizations of all short exact sequences in $\Fact(\scrE,\Phi,w)$ and closed under coproducts. Then we define the {\it coderived factorization category} $\Dco(\scrE,\Phi,w)$ by
\begin{gather*}
\Dco(\scrE,\Phi,w)\defeq\K(\scrE,\Phi,w)/\Aco(\scrE,\Phi,w).
\end{gather*}
\end{dfn}

Let $\scrF$ be an exact subcategory of another abelian category $\scrB$ satisfying the same properties of $\scrA$, and let
\begin{gather*}
(\scrF,\Psi,v)
\end{gather*} be a triple as in \eqref{triple}.

\begin{dfn} An additive functor
\begin{gather*}
F\colon\ \scrE \to \scrF
\end{gather*}
is {\it factored} with respect to the potentials $(\Phi,w)$ and $(\Psi,v)$ if there is a functor isomorphism
\begin{gather*}
\alpha\colon\ F\circ \Phi \simto \Psi\circ F
\end{gather*}
such that for every object $E\in \scrE$, the following diagram commutes:
\[
\begin{tikzcd}
&F(E)\arrow[ld, "F(w(E))"']\arrow[rd, "v(F(E))"]&\\
F(\Phi(E))\arrow[rr, "\alpha(E)"]&& \Psi(F(E)).
\end{tikzcd}
\]
\end{dfn}

If $F\colon \scrE \to \scrF$ is a factored functor with respect to $(\Phi,w)$ and $(\Psi,v)$, then it induces a dg functor
\begin{gather*}
F\colon\ \dgFact(\scrE,\Phi,w)\to\dgFact(\scrF,\Psi,v).
\end{gather*}
This dg functor defines the additive functors
\begin{gather*}
F\colon\ \Fact(\scrE,\Phi,w)\to\Fact(\scrF,\Psi,v),
\\
F\colon\ \K(\scrE,\Phi,w)\to\K(\scrF,\Psi,v),
\end{gather*}
where the latter functor is an exact functor. To define the derived functors of exact functors between homotopy categories, we need the following results due to~\cite{bdfik}.

\begin{prop}[{\cite[Corollary~2.25]{bdfik}}]\label{inj resol}
Assume that $\scrA$ has enough injectives and that the coproducts of injectives are injective. Let $\scrI\subset \scrA$ be the subcategory of injective objects. Then the natural functor
\begin{gather*}
\K(\scrI,\Phi,w)\to\Dco(\scrA,\Phi,w)
\end{gather*}
is an equivalence.
\end{prop}

\begin{prop}\label{proj resol}
Let $\scrP\subset \scrA$ be the subcategory of projective objects in $\scrA$, and assume that~$\scrA$ has enough projectives. Let $\scrC\subset \scrA$ be an abelian subcategory that is preserved by $\Phi$, and let $\scrQ\defeq\scrC\cap \scrP$.
\begin{itemize}\itemsep=0pt
\item[$1.$] Assume that all objects in $\scrC$ are compact in $\scrA$. Then for any $P\in \K(\scrQ,\Phi,w)$ and $A\in \Aco(\scrA,\Phi,w)$ we have
\begin{gather*}
\Hom_{\K(\scrA,\Phi,w)}(P,A)=0.
\end{gather*}
In particular,
the natural functor
\begin{gather*}
\K(\scrQ,\Phi,w)\to\Dco(\scrA,\Phi,w)
\end{gather*}
is fully faithful.

\item[$2.$] If every object in $\scrC$ has a finite projective resolution in $\scrC$, the natural functor
\begin{gather*}
\K(\scrQ,\Phi,w)\to\D(\scrC,\Phi,w)
\end{gather*}
is essentially surjective.
\end{itemize}
\begin{proof}
(1) The latter statement follows from the former one by~\cite[Proposition~B.2]{ls}. The~vani\-shing $\Hom_{\K(\scrA,\Phi,w)}(P,A)=0$ reduces to the case when $A\in \A(\scrA,\Phi,w)$, since the components of~$P$ are compact by our assumption. If $A\in \A(\scrA,\Phi,w)$, the vanishing follows from~\cite[Lemma~2.4]{bdfik}.

(2) This follows from~\cite[Proposition~2.22]{bdfik}.
\end{proof}
\end{prop}

\begin{dfn} Let $F\colon \scrA\to \scrB$ be a factored functor with respect to $(\Phi,w)$ and $(\Psi,v)$.

\begin{itemize}\itemsep=0pt
\item[1.] Assume that $\scrA$ has enough injectives and that the coproducts of injectives are injective.
 If $F$ is left exact, we define the {\it right derived functor}
\begin{gather*}
\bR F\colon\ \Dco(\scrA,\Phi,w)\to \Dco(\scrB,\Psi,v)
\end{gather*}
of $F\colon \K(\scrA,\Phi, w)\to\K(\scrB,\Psi,v)$ to be the composition
\[
\begin{tikzcd}
\Dco(\scrA,\Phi,w)\arrow[r,"\sim"]&\K(\scrI,\Phi,w)\arrow[r,"F"]&\K(\scrB,\Psi,v)\arrow[r,"Q"]&\Dco(\scrB,\Psi,v),
\end{tikzcd}
\]
where the first functor is the equivalence in Proposition~\ref{inj resol}, and $Q$ is the natural quotient functor.

 \item[2.] Let $\scrC\subset \scrA$ and $\scrD\subset \scrB$ be abelian subcategories that are preserved by $\Phi$ and $\Psi$ respectively, and assume that the factored functor $F$ restricts to $F\colon \scrC\to \scrD$. Assume that every object in $\scrC$ is compact in $\scrA$ and has a finite projective resolution in $\scrC$, and that the natural functor $\D(\scrC,\Phi,w)\to \Dco(\scrA,\Phi,w)$ is fully faithful.
If $F$ is right exact, we define the {\it left derived functor}
\begin{gather*}
\bL F\colon\ \D(\scrC,\Phi,w)\to \D(\scrD,\Psi,v)
\end{gather*}
of $F\colon \K(\scrC,\Phi,w)\to \K(\scrD,\Psi,v)$ to be the composition
\[
\begin{tikzcd}
\D(\scrC,\Phi,w)\arrow[r,"\sim"]&\K(\scrQ,\Phi,w)\arrow[r,"F"]&\K(\scrD,\Psi,v)\arrow[r,"Q"]&\D(\scrD,\Psi,v),
\end{tikzcd}
\]
where the first functor is the equivalence in Proposition~\ref{proj resol}.
\end{itemize}
\end{dfn}

\subsection{Quasi-coherent modules over sheaves of algebras}

In this subsection, we recall fundamental properties of quasi-coherent modules over sheaves of~algebras over schemes.

Let $X$ be a scheme. An $\cO_X$-module $\cF$ is said to be {\it quasi-coherent} if for any point $x\in X$, there are an open neighborhood $U$ of $x$
and an exact sequence
\begin{gather*}
\cO_U^{\oplus I}\to \cO_U^{\oplus J}\to \cF|_U\to0,
\end{gather*}
where $I$ and $J$ are sets. A quasi-coherent $\cO_X$-module $\cF$ is called a {\it coherent} $\cO_X$-module, if for every affine open subset $U=\Spec R$ of $X$, $\Gamma(U,\cF)$ is a finitely generated $R$-module. We denote by $\Qcoh X$ (resp.\ $\coh X$) the category of quasi-coherent (resp.\ coherent) $\cO_X$-modules.

\begin{dfn}
An {\it $X$-algebra} is a sheaf $\cA$ of not necessarily commutative algebras together with a morphism $\uprho\colon \cO_X\to \cA$ of sheaves of algebras such that the image of $\uprho$ lies in the center of~$\cA$ and that $\cA$ is a quasi-coherent $\cO_X$-module. An $X$-algebra $\uprho\colon \cO_X\to \cA$ is said to be {\it coherent} if $\cA$ is a coherent $\cO_X$-module. A {\it morphism} of $X$-algebras is a morphism $\varphi\colon \cA\to \cB$ of~sheaves of algebras that is $\cO_X$-linear.

For an $X$-algebra $\uprho\colon \cO_X\to\cA$, a {\it quasi-coherent $($resp.\ coherent$)$ $\cA$-module} is a right $\cA$-mo\-dule~$\cM$ that is a quasi-coherent (resp.\ coherent) $\cO_X$-module. For quasi-coherent $\cA$-modu\-les~$\cM$ and~$\cN$, a {\it morphism} from $\cM$ to $\cN$ is a morphism of sheaves of right $\cA$-modules. We~denote~by\vspace{-.7ex}
\begin{gather*}
\Qcoh \cA
\end{gather*}
the category of quasi-coherent $\cA$-modules, and we write $\coh \cA$ for the full subcategory of coherent $\cA$-modules.
\end{dfn}

\begin{rem}
Since the category of right $\cA$-modules and the category $\Qcoh X$ are both abelian, $\Qcoh \cA$ is also an abelian category. Furthermore, in Proposition~\ref{grothen prop} we will see that $\Qcoh \cA$ is~a~Grothendieck category.
\end{rem}

Let $\varphi\colon \cA\to \cB$ be a morphism of $X$-algebras. For an object $\cN\in \Qcoh \cB$, we define a~qua\-si-coherent $\cA$-module $\cN_{\varphi}$ by the composition\vspace{-.7ex}
 \begin{gather*}
 \cN\times \cA\xrightarrow{\id\times \varphi} \cN\times \cB\to \cN,
 \end{gather*}
 where the second morphism is the right $\cB$-action of $\cN$. This defines the functor\vspace{-.7ex}
\begin{equation}\label{restriction}
(-)_{\varphi}\colon\ \Qcoh \cB\to \Qcoh \cA,
\end{equation}
and we call the functor $(-)_{\varphi}$ the {\it restriction by $\varphi$}.

Conversely, for an object $\cM\in \Qcoh \cA$, the tensor product
$\cM\otimes_{\cA}\cB$ has a natural right $\cB$-action, and this defines the functor\vspace{-.7ex}
\begin{equation}\label{extension}
(-)\otimes_\cA\cB\colon\ \Qcoh \cA\to \Qcoh \cB,
\end{equation}
and we call this functor the {\it extension by $\varphi$}.
It is standard that we have an adjunction\vspace{-.7ex}
\begin{gather*}
(-)\otimes_\cA\cB \dashv (-)_{\varphi}.
\end{gather*}

Although the following propositions might be well known, we give the proofs for reader's convenience.
\begin{prop}
Let $X$ be a scheme, and $\uprho\colon \cO_X\to\cA$ an $X$-algebra.\vspace{-.5ex}
\begin{itemize}\itemsep=-1pt
\item[$1.$] For every $\cM\in \Qcoh \cA$ and every point $x\in X$, there are an open neighborhood $U$ of $x$ and a surjective $\cA|_U$-linear map\vspace{-.7ex}
\begin{gather*}
\cA|_U^{\oplus J}\surjto \cM|_U,
\end{gather*}
where $J$ is a set.
\item[$2.$] A right $\cA$-module $\cM$ is quasi-coherent if and only if for every point $x\in X$ there are an~open neighborhood $U$ and an exact sequence\vspace{-.7ex}
 \begin{gather*}
 \cA|_U^{\oplus I}\to \cA|_U^{\oplus J}\to \cM|_U\to 0.
 \end{gather*}
 in the category of right $\cA$-modules.
\end{itemize}
 \begin{proof}
 (1) Let $\cM\in \Qcoh \cA$ and $x\in X$. Then there are an open neighborhood $U$ of $x$, a set $J$ and an surjective morphism
 $\pi\colon \cO_U^{\oplus J}\surjto\cM_{\uprho}|_U$
 in $\Qcoh X$. Since the functor $(-)\otimes_{\cO_U}\cA|_U\colon \Qcoh U\to \Qcoh \cA|_U$ has a right adjoint functor, it is right exact and commutes with small direct sums. Hence we have a natural isomorphism $\cO_U^{\oplus J}\otimes _{\cO_U}\cA_U\cong \cA|_U^{\oplus J}$ and the morphism
{\samepage \begin{gather*}
 \pi\otimes_{\cO_U}\cA|_U\colon\ \cA_U^{\oplus J}\to \cM_{\uprho}|_U\otimes_{\cO_U}\cA|_U
 \end{gather*}
 is surjective. Composing this surjection with a natural surjective morphism
 $\cM_{\uprho}|_U\otimes_{\cO_U}\cA|_U\surjto \cM|_U$; $m\otimes a\mapsto ma$, we have a surjective morphism $\cA|_U^{\oplus J}\surjto \cM|_U$.}

 (2) ($\Rightarrow$) Assume that a right $\cA$-module $\cM$ is quasi-coherent and let $x\in X$ be a point. Then by (1) there are an open neighbborhood $U$ of $x$ and a surjective $\cA|_U$-linear map $p \colon \cA|_U^{\oplus J}\surjto \cM|_U$. Denote by $\cK$ the Kernel of $p$ and by $i\colon \cK\hookto \cA|_U^{\oplus J}$ the natural inclusion. Then $\cK$ is also a quasi-coherent $\cA$-module, and thus, by shrinking $U$ if necessary, we have a surjective $\cA|_U$-linear map $q \colon \cA|_U^{\oplus I}\surjto \cK|_U$. Then the sequence
 \begin{gather*}
 \cA|_U^{\oplus I}\xrightarrow{i \circ q}\cA|_U^{\oplus J}\xrightarrow{p}\cM|_U\to 0
 \end{gather*}
is an exact.

($\Leftarrow$) Assume that a right $\cA$-module $\cM$ is locally isomorphic to the cokernel of a morphism between free modules. Then, since $\cA$ is a quasi-coherent $\cO_X$-module and $\cM$ is locally isomorphic to the cokernel of a morphism of quasi-coherent modules, $\cM$ is a locally quasi-coherent module. Thus $\cM$ is a quasi-coherent $\cO_X$-module.
 \end{proof}
\end{prop}

\subsection{Equivariant sheaves}\label{equiv section}
We briefly recall the basics of equivariant quasi-coherent sheaves. For more details, see, for example,~\cite[Section~2]{bfk}.

Let $G$ be an algebraic group, and denote by $\varepsilon\colon \Spec k\to G$, $\mu\colon G\times G\to G$ and $\iota\colon G\to G$ the morphisms defining the identity, the multiplication and the inversion of $G$.
Let $X$ be a scheme with an algebraic left $G$-action $\sigma\colon G\times X \to X$. We denote by $\pi\colon G\times X\to X$ the natural projection, and we write
\begin{gather*}
\varepsilon_X\colon\ X\to G\times X
\end{gather*}
for the composition $X\simto \Spec k\times X\xrightarrow{ \varepsilon\times \id_X}G\times X$. \begin{dfn}
A {\it $G$-equivariant structure} on a quasi-coherent sheaf $\cF\in \Qcoh X$ is an isomorphism
\begin{gather*}
\uptheta\colon\ \pi^*\cF\simto \sigma^*\cF
\end{gather*}
of $\cO_{G\times X}$-modules such that the equations
\begin{equation}\label{compati}
\varphi^*\uptheta\circ (\id_G\times\pi)^*\uptheta=(\mu\times\id_X )^*\uptheta \qquad\text{and}\qquad
\varepsilon_X^*\uptheta=\id_{\cF}
\end{equation}
hold, where $\varphi\colon G\times G\times X\to G\times X$ is the morphism defined by $\varphi(g,h,x)\defeq(h,gx)$.
\end{dfn}

\begin{rem}\label{closed point isom}
For any closed point $g\in G$, there is the corresponding morphism $g\colon \Spec k\to G$, and this induces the morphism $g\times\id_X\colon X\to G\times X$. The pull-back of a $G$-equivariant structure $\uptheta\colon \pi^*\cF\simto \sigma^*\cF$ by $g\times \id_X$ defines the isomorphism
\begin{gather*}
\uptheta_g\colon\ \cF\simto \sigma_g^*\cF,
\end{gather*}
where $\sigma_g\colon X\to X$ is the isomorphism defined by $\sigma_g\defeq \sigma\circ (g\times \id_X)$. The pull-back of the first equation in (\ref{compati}) by the morphism $ g\times h\times\id_X\colon X\to G\times G\times X$ implies the equation
\begin{gather*}
\sigma_g^*\uptheta_h\circ \uptheta_g=\uptheta_{gh},
\end{gather*}
and the pull-back of the second equation is nothing but the equation $\uptheta_{1_G}=\id_{\cF}$.
\end{rem}

A {\it $G$-equivariant quasi-coherent sheaf} \,is a pair $(\cF, \uptheta)$ of a quasi-coherent sheaf $\cF\in\Qcoh X$ and a $G$-equivariant structure $\uptheta$ on $\cF$. We say that $(\cF,\uptheta)$ is {\it coherent} (resp.\ {\it locally free}) if $\cF$ is coherent (resp.\ locally free). If there is no risk of confusion, a $G$-equivariant quasi-coherent sheaf $(\cF, \uptheta)$ is denoted simply by $\cF$.

Let $\big(\cF,\uptheta^{\cF}\big)$ and $\big(\cG,\uptheta^{\cG}\big)$ be $G$-equivariant quasi-coherent sheaves on $X$. A morphism $\varphi$: $\cF\to \cG$ of $\cO_X$-modules is said to be {\it $G$-equivariant} if the diagram
\begin{equation}\label{diag}
\begin{tikzcd}
\pi^*\cF\arrow[rr,"\pi^*\varphi"]\arrow[d,"\uptheta^{\cF}"']&&\pi^*\cG\arrow[d," \uptheta^{\cG}"]\\
\sigma^*\cF\arrow[rr, "\sigma^*\varphi"]&&\sigma^*\cG
\end{tikzcd}\end{equation}
is commutative.

\begin{rem}\label{cl pt}
Let $\big(\cF,\uptheta^{\cF}\big)$ and $\big(\cG,\uptheta^{\cG}\big)$ be $G$-equivariant quasi-coherent sheaves on $X$. Assume that $X$ is of finite type over $k$. Then the set of closed points of $G\times X$ is equal to the set of pairs $(x,g)$ of closed points $x\in X$ and $g\in G$. This implies that a morphism $\varphi\colon\cF\to \cG$ of the $\cO_X$-modules is $G$-equivariant if and only if for any closed point $g\in G$ the pull-back
\begin{equation*}
\begin{tikzcd}
\cF\arrow[rr,"\varphi"]\arrow[d,"\uptheta^{\cF}_g"']&&\cG\arrow[d," \uptheta^{\cG}_g"]\\
\sigma_g^*\cF\arrow[rr, "\sigma^*_g\varphi"]&&\sigma_g^*\cG
\end{tikzcd}
\end{equation*}
of the diagram (\ref{diag}) by $g\times\id_X $ is commutative.
\end{rem}

$G$-equivariant quasi-coherent sheaves on $X$ and $G$-equivariant morphisms define the abelian category
\begin{gather*}
\Qcoh_GX.
\end{gather*}
It is standard that the category $\Qcoh_GX$ is equivalent to the category $\Qcoh [X/G]$ of quasi-coherent sheaves on the quotient stack $[X/G]$;
\begin{equation}\label{equiv stack}
\Qcoh_GX\cong \Qcoh [X/G],
\end{equation}
 and for $\cF\in \Qcoh_GX$ we denote by $[\cF/G]\in \Qcoh [X/G]$ the image of $\cF$ by the equivalence \eqref{equiv stack}. We denote by $\coh_GX$ the subcategory of $G$-equivariant coherent sheaves on $X$. This category is an exact subcategory, and if $X$ is noetherian, it is an abelian subcategory. Furthermore, if $X$ is of finite type over $k$, by Remark~\ref{cl pt}, we have
\begin{equation}\label{G hom}
\Hom_{\Qcoh_GX}\big(\big(\cF,\uptheta^{\cF}\big),\big(\cG,\uptheta^{\cG}\big)\big) =\Hom_{\Qcoh X}(\cF,\cG)^G,
\end{equation}
where the right hand side $\Hom_{\Qcoh X}(\cF,\cG)^G$ is the $G$-invariant subset with respect to the $G$-action on $\Hom_{\Qcoh X}(\cF,\cG)$ defined by
$g\cdot\varphi\defeq (\uptheta^{\cG}_g)^{-1}\circ\sigma_g^*\varphi\circ\uptheta^{\cF}_g$.

Let $Y$ be another $G$-scheme, and $f\colon X\to Y$ a $G$-equivariant morphism that is quasi-compact and quasi-separated. Then $f$ induces the direct image\vspace{-.5ex}
\begin{gather*}
f_*\colon\ \Qcoh_G X\to \Qcoh_G Y
\end{gather*}
and the inverse image\vspace{-.5ex}
\begin{gather*}
f^*\colon\ \Qcoh_G Y\to \Qcoh_GX.
\end{gather*}
It is standard that $f^*$ is a left adjoint functor of $f_*$.

Tensor products and sheaf Hom define bi-functors
\begin{gather*}
(-)\otimes_{X}(-)\colon\ \Qcoh_G X\times \Qcoh_G X\to \Qcoh_G X,
\\[.5ex]
\cHom(-,-)\colon\ (\coh_G X)^{\rm op}\times \Qcoh_G X\to \Qcoh_G X,
\end{gather*}
and it is also standard that for any $\cF\in \coh_GX$, the functor $\cHom(\cF,-)\colon\Qcoh_G X\to \Qcoh_G X$ is right adjoint to the functor $(-)\otimes_{X}\cF\colon\Qcoh_G X\to \Qcoh_G X$.

\subsection{Tilting objects and derived equivalences}

We recall that tilting objects induce derived equivalences.
Throughout this subsection, $\scrA$ is an abelian category with small coproducts and enough injectives, and we assume that small coproducts of families of short exact sequences in $\scrA$ are exact. These properties of $\scrA$ are satisfied if $\scrA$ is a Grothendieck category.

\begin{dfn} Let $T\in \scrA$ be an object. The object $T$ is called a {\it tilting object} if the following conditions are satisfied:
\begin{itemize}\itemsep=0pt
\item[$1.$] $\Ext_{\scrA}^i(T,T)=0$ for all $i>0$.

\item[$2.$] $T$ is {\it compact} in $\D(\scrA)$, i.e., the natural map
\begin{gather*}
\bigoplus_{i\in \cI}\Hom_{\D(\scrA)}(T,A_i^{\bullet})\to\Hom_{\D(\scrA)}\Big(T,\bigoplus_{i\in\cI}A_i^{\bullet}\Big)
\end{gather*}
is an isomorphism of abelian groups for any set $\cI$ and any family $\{A_i^{\bullet}\}_{i\in I}$.

\item[$3.$] $T$ is a {\it generator} of $\D(\scrA)$, i.e., for an object $A^{\bullet}\in \D(\scrA)$, \[\RHom(T,A^{\bullet})\cong0$ in $\D\big(\Mod\End_{\scrA}(T)\big)\] implies $A^{\bullet}\cong0$ in $\D(\scrA)$.
\end{itemize}
We call $T$ a {\it partial tilting object} if the conditions (1) and (2) are satisfied.
\end{dfn}

For an object $T\in \scrA$, we denote by
\begin{gather*}
\Add T
\end{gather*}
the smallest additive subcategory containing all direct summands of small coproducts of $T$. We~write $\add T\subset \Add T$ for the subcategory consisting of direct summands of finite coproducts of~$T$. If~$T$ is a partial tilting object, then for arbitrary sets $\cI$ and $\cJ$, we have isomorphisms
\begin{align*}
\Ext^n_{\scrA}\Big(\bigoplus_{i\in \cI}T,\bigoplus_{j\in \cJ}T\Big) &\cong\Hom_{\D(\scrA)}\Big(\bigoplus_{i\in \cI}T,\bigoplus_{j\in \cJ}(T[n])\Big)
\\
&\cong \prod_{i\in \cI}\Hom_{\D(\scrA)}\Big(T,\bigoplus_{j\in \cJ}(T[n])\Big)
\\
&\cong \prod_{i\in\cI}\Big(\bigoplus_{j\in\cJ}\Hom_{\D(\scrA)}(T,T[n])\Big).
\end{align*}
This implies that for any $X,Y\in \Add T$, we have $\Ext_{\scrA}^n(X,Y)=0$ for any $n>0$. Thus any short exact sequence $0\to X\to Z\to Y\to0$ in $\scrA$ with $X,Y\in \Add T$ splits, and in particular $\Add T$ is an exact subcategory of $\scrA$.

In the remaining of this subsection, $T\in\scrA$ is a partial tilting object, and we assume that every right module over the endomorphism ring $\Lambda\defeq\End_{\scrA}(T)$ is of finite projective dimension. Consider the functor
 \begin{gather*}
{\F}\defeq\Hom_{\scrA}(T,-)\colon\ \scrA\to\Mod\Lambda
 \end{gather*}
and assume that
\begin{itemize}\itemsep=0pt
 \item[$(i)$] $\F$ has a left adjoint functor $\G\colon\Mod\Lambda\to \scrA$ such that the adjunction morphisms $\sigma({\Lambda})\colon \Lambda\to \F\G(\Lambda)$ and $\tau(T)\colon \G\F(T)\to T$ are isomorphisms.

 \item[$(ii)$] The right derived functor $\bR \F\colon \D^+(\scrA)\to \D^+(\Mod\Lambda)$ restricts to the functor
\begin{gather*}
 \bR \F\colon\ \Db(\scrA)\to \Db(\Mod\Lambda)
 \end{gather*}
 between bounded derived categories.
 \end{itemize}

 Note that both of the functors $\F$ and $\G$ commute with small coproducts since $T$ is compact and $\G$ admits a right adjoint functor. In particular, $\F$ and $\G$ restrict to the following functors
 \begin{gather*}
 \F\colon\ \Add T\to\Proj\Lambda,
 \\
\G\colon\ \Proj\Lambda\to \Add T,
 \end{gather*}
where $\Proj \Lambda$ is the category of projective right $\Lambda$-modules.
\begin{lem}\label{add proj}
The functor $\F\colon\Add T\to\Proj\Lambda$ is an equivalence, and $\G\cong \F^{-1}$.
\begin{proof}
The adjunction morphisms $\sigma({P})\colon P\to \F\G(P)$ and $\tau(S)\colon \G\F(S)\to S$ are bijective for any $P\in\Mod\Lambda$ and $S\in \Add T$, since $\sigma({\Lambda})$ and $\tau({T})$ are bijective and $\Proj \Lambda=\Add \Lambda$. Hence, the functors $\F$ and $\G$ are equivalences, and $\G\cong \F^{-1}$.
\end{proof}
\end{lem}

By our assumptions, the functor $\F$ defines the right derived functor
\begin{gather*}
\bR\F\colon\ \Db(\scrA)\to \Db(\Mod\Lambda)
\end{gather*}
between bounded derived categories, and the functor $\G$ also defines the left derived functor
\begin{gather*}
\bL\G\colon\ \Db(\Mod\Lambda)\to\Db(\scrA)
\end{gather*}
between bounded derived categories.
The following is well known to experts.

\begin{thm}\label{general tilting}
The functor $\bL \G\colon \Db(\Mod\Lambda)\to\Db(\scrA)$
is fully faithful. Furthermore, if $T$ is a tilting object, then $\bL \G$ is an equivalence and $\bR \F\cong (\bL \G)^{-1}$.
\begin{proof}
 We have the following commutative diagram
\[
\begin{tikzcd}
\Kb(\Proj\Lambda)\arrow[rr,"\G"]\arrow[d,]&&\Kb(\Add T)\arrow[d,]\\
\Db(\Mod\Lambda)\arrow[rr, "\bL \G\,\,\,\,\,\,"]&&\Db(\scrA),
\end{tikzcd}
\]
where the vertical arrows are the natural quotient functors. The left vertical functor is an equi\-va\-lence, since every right $\Lambda$-module has a finite projective resolution. Moreover, by Lemma~\ref{add proj}, the top horizontal functor is also an equivalence. Hence, for the former assertion, it is enough to show that the natural functor
\begin{gather*}
\Kb(\Add T)\to \Db(\scrA)
\end{gather*}
is fully faithful, and this follows from an identical argument as in~\cite[Chapter~III, Lemma~2.1]{happel}.

For the latter assertion, assume that $T$ is a tilting object. It is enough to prove that $\bR \F$ is also fully faithful. For any $A^{\bullet}\in \D(\scrA)$, consider the following triangle
\begin{gather*}
\bL \G \bR \F(A^{\bullet})\xrightarrow{\tau}A^{\bullet}\to C(\tau)\to \bL \G \bR \F(A^{\bullet})[1],
\end{gather*}
where $\tau$ is the adjunction morphism. Applying the functor $\bR \F$ to the above triangle and using the natural isomorphism $\bR \F\circ \,\bL \G\cong \id$, we see that $\bR \F(C(\tau))\cong 0$ in $\Db(\Mod\Lambda)$. Since $T$ is a~generator of $\D(\scrA)$, this implies that $C(\tau)\cong 0$ in $\D(\scrA)$. Hence the adjunction morphism $\tau$ is an~isomorphism. This completes the proof.
\end{proof}
\end{thm}

\begin{rem}
Our assumption that every $\Lambda$-module has a finite projective resolution might be weakened. However, the result under our assumption is enough for our purpose.
\end{rem}

\section{Equivariant algebras and equivariant tilting modules}

In this section, we introduce equivariant algebras and equivariant modules, and define seve\-ral functors between equivariant modules that are generalizations of functors in~\cite[Section~2]{bfk}. We~also show that equivariant tilting modules induce derived equivalences of equivariant modules, which are simultaneous generalizations of tilting equivalences induced by tilting bundles on schemes and tilting modules over noncommutative algebras.

\subsection{Equivariant modules over equivariant algebras}

In this subsection, we give the definitions of equivariant algebras and equivariant modules, and~dis\-cuss basic properties.
Notation is the same as in Section~\ref{equiv section}. Let $\uprho\colon \cO_X\to \cA$ be an $X$-algebra.

\begin{dfn}
A {\it $G$-equivariant structure} on the $X$-algebra $\uprho\colon \cO_X\to\cA$ is an isomorphism
\begin{gather*}
\uptheta^{\cA}\colon\ \pi^*\cA\simto \sigma^*\cA
\end{gather*}
 of sheaves of algebras on $G\times X$
such that it is $\cO_{G\times X}$-linear and it gives a $G$-equivariant structure on the $\cO_X$-module $\cA$. The pair $\big(\cA,\uptheta^{\cA}\big)$ is called a {\it $G$-equivariant algebra over $X$}, or {\it $G$-equivariant $X$-algebra}. We say that $\big(\cA,\uptheta^{\cA}\big)$ is {\it coherent} if the underlying $X$-algebra $\cA$ is coherent $\cO_X$-module. We write simply $\cA$ for $\big(\cA,\uptheta^{\cA}\big)$ if no confusion seems likely to occur.
\end{dfn}

\begin{rem}\label{linearity}
Note that for an isomorphism $\upphi \colon \pi^*\cA\simto \sigma^*\cA$ of sheaves of algebras which defines a $G$-equivariant structure on the quasi-coherent sheaf $\cA$, $\upphi$ is $\cO_{G\times X}$-linear if and only if the morphism $\uprho\colon \cO_X\to \cA$ is a $G$-equivariant morphism of $G$-equivariant quasi-coherent sheaves $(\cO_X,\uptheta^{\rm can})$ and $(\cA,\upphi)$, where $\uptheta^{\rm can}$ is the canonical equivariant structure on the structure sheaf~$\cO_X$ induced by the $G$-action on $X$.
\end{rem}

Let $X=\Spec R$ be an affine scheme with an action from an affine algebraic group $G$, and set $R_G\defeq k[G]\otimes_kR$, where $k[G]$ denotes the coordinate ring of $G$. We denote by
\begin{gather*}
R_G^{\pi}\qquad \big(\text{resp.}~R_G^{\sigma}\big)
\end{gather*}
the $R$-algebra $R\to R_G$ induced by the morphism $\pi\colon G\times X\to X$ (resp.\ $\sigma\colon G\times X\to X$).
If~$\big(\cA,\uptheta^{\cA}\big)$ is a $G$-equivariant $X$-algebra, taking global sections induces an $R$-algebra $A\defeq\cA(X)$ and an isomorphism
\begin{gather*}
\uptheta^{A}\defeq\uptheta^{\cA}(X)\colon\ A\otimes_R R_G^{\pi}\simto A\otimes_RR_G^{\sigma}
\end{gather*}
of $R_G$-algebras such that $\uptheta^{A}$ gives a $G$-equivariant structure on the $R$-module $A$. This yields the following definition.

 \begin{dfn}
 We call a pair
 $(\Lambda,\uptheta^{\Lambda})$
 a {\it $G$-equivariant $R$-algebra}, if $\Lambda$ is an $R$-algebra and $\uptheta^{\Lambda}\colon \Lambda\otimes_RR_G^{\pi}\simto\Lambda\otimes_RR_G^{\sigma}$ is an isomorphism of $R_G$-algebras such that it defines a $G$-equivariant structure on the $R$-module $\Lambda$.
 \end{dfn}

 \begin{rem}
For any commutative ring $S$, via the natural equivalence $\Mod S\simto \Qcoh \Spec S$, to give a $G$-equivariant algebra over $\Spec S$ is equivalent to give a $G$-equivariant $S$-algebra.
 \end{rem}

\begin{exa}\label{end ex}
Notation is the same as above.
\begin{enumerate}\itemsep=0pt
\item Let $\cF\in \coh_GX$ be a $G$-equivariant coherent sheaf on $X$. Then the algebra
\begin{gather*}
\cA\defeq\cEnd_X(\cF)
\end{gather*}
over $X$ has a natural $G$-equivariant structure. Indeed, the $G$-equivariant structure
\begin{gather*}
\uptheta\colon\ \pi^*\cEnd_X(\cF)\simto \sigma^*\cEnd_X(\cF)
\end{gather*}
on the $G$-equivariant coherent sheaf $\cEnd_X(\cF)\in\coh_GX$
is a morphism of algebras, and so the pair $(\cA, \uptheta$) is a $G$-equivariant algebra over $X$.

\item Let $\big(\cA,\uptheta^{\cA}\big)$ be a $G$-equivariant $X$-algebra. If $Y$ is another $G$-scheme and $f\colon X\to Y$ is a~$G$-equ\-i\-variant morphism that is quasi-separated and quasi-compact,
then the direct image
$\big(f_*\cA,(\id_G\times f)_*\uptheta^{\cA}\big)$
is a $G$-equivariant $Y$-algebra. If $\cA$ is coherent and $f$ is proper, the algebra $f_*\cA$ is also coherent. Similarly, if $g\colon Y\to X$ is a $G$-equivariant morphism, then the pull-back $\big(g^*\cA,(\id_G\times g)^*\uptheta^{\cA}\big)$ is a $G$-equivariant $Y$-algebra, and it is coherent if~$\cA$ is coherent.
 \end{enumerate}
\end{exa}

We define equivariant modules over equivariant algebras.

\begin{dfn}\label{equiv sheaf}
Let $\big(\cA,\uptheta^{\cA}\big)$ be a $G$-equivariant $X$-algebra.
\begin{itemize}\itemsep=0pt
\item[1.] A {\it $G$-equivariant structure} on a quasi-coherent right $\cA$-module $\cM$ is an isomorphism
\begin{gather*}
\uptheta^{\cM}\colon\ \pi^*\cM\simto \sigma^*\cM
\end{gather*}
of $\pi^*\cA$-modules such that $\uptheta^{\cM}$ satisfies the condition \eqref{compati}, where the $\pi^*\cA$-module structure on the $\sigma^*\cA$-module $\sigma^*\cM$ is given by $\uptheta^{\cA}$.
We call the pair $\big(\cM,\uptheta^{\cM}\big)$ a {\it $G$-equivariant quasi-coherent right $\cA$-module}, or simply {\it $G$-equivariant $\cA$-module}. If the underlying sheaf~$\cM$ of~modu\-les is coherent, we call $\big(\cM,\uptheta^{\cM}\big)$ a {\it $G$-equivariant coherent $\cA$-module}. We~sometimes write just $\cM$ for $\big(\cM,\uptheta^{\cM}\big)$.

\item[2.] Let $\big(\cM,\uptheta^{\cM}\big)$ and $\big(\cN,\uptheta^{\cN}\big)$ be $G$-equivariant $\cA$-modules. A morphism $\varphi\colon \cM\to \cN$ of $\cA$-modules is {\it $G$-equivariant} if the following diagram
\[
\begin{tikzcd}
\pi^*\cM\arrow[rr,"\pi^*\varphi"]\arrow[d,"\uptheta^{\cM}"']&&\pi^*\cN\arrow[d," \uptheta^{\cN}"]\\
\sigma^*\cM\arrow[rr, "\sigma^*\varphi"]&&\sigma^*\cN
\end{tikzcd}
\]
is commutative.

\item[3.] We denote by
\begin{gather*}\Qcoh_G\cA\end{gather*} the category of $G$-equivariant $\cA$-modules whose morphisms are $G$-equivariant, and write $\coh_G\cA$ for the full subcategory of $G$-equivariant coherent $\cA$-modules.
\end{itemize}
\end{dfn}

\begin{rem}
(1) If $X$ is a $G$-scheme, then $\cO_X$ is a natural $G$-equivariant algebra over $X$. For any $(\cF,\uptheta)\in \Qcoh_GX$, the $G$-equivariant structure $\uptheta$ is automatically $\pi^*\cO_X$-linear. Therefore, we have a natural identification
\begin{gather*}
\Qcoh_G\cO_X=\Qcoh_GX.
\end{gather*}
Hence $G$-equivariant modules are generalizations of $G$-equivariant quasi-coherent sheaves.
\end{rem}

\begin{dfn}
Let $X=\Spec R$ be an affine scheme with an action from an affine algebraic group $G$, and $(\Lambda,\uptheta^{\Lambda})$ a $G$-equivariant $R$-algebra. A {\it $G$-equivariant structure} on a right $\Lambda$-module~$M$ is an isomorphism
\begin{gather*}
\uptheta^M\colon\ M\otimes_RR_G^{\pi}\simto M\otimes_RR_G^{\sigma}
\end{gather*}
of $\big(\Lambda\otimes_RR_G^{\pi}\big)$-modules, where $\big(\Lambda\otimes_RR_G^{\sigma}\big)$-module $M\otimes_RR_G^{\sigma}$ is considered as a $\big(\Lambda\otimes_RR_G^{\pi}\big)$-mo\-dule via the ring isomorphism $\uptheta^{\Lambda}\colon \Lambda\otimes_RR_G^{\pi}\simto\Lambda\otimes_RR_G^{\sigma}$. We call such a pair $\big(M,\uptheta^M\big)$ a~{\it $G$-equivariant $\Lambda$-module}. We denote by
\begin{gather*}
\Mod_G\Lambda
\end{gather*}
the category of $G$-equivariant $\Lambda$-modules whose morphisms are defined similarly to Definition~\ref{equiv sheaf}, and we denote by
\begin{gather*}
\fmod_G\Lambda\subset \Mod_G\Lambda
\end{gather*}
the full subcategory consisting of equivariant modules that are finitely generated over $\Lambda$.
\end{dfn}

\begin{rem}
Let $X=\Spec R$ be an affine scheme with an action from an affine algebraic group $G$, and $\big(\cA,\uptheta^{\cA}\big)$ a $G$-equivariant $X$-algebra. Then it induces a $G$-equivariant $R$-algebra $A\defeq\cA(X)$, and we have a natural equivalence
\begin{gather*}
\Qcoh_G\cA\simto\Mod_GA
\end{gather*}
of abelian categories, which restricts to an equivalence $\coh_G\cA\simto\fmod_GA$.
\end{rem}
\vspace{1mm}

Let $\big(\cM,\uptheta^{\cM}\big)$ be a $G$-equivariant $\cA$-module. Recall from Remark~\ref{closed point isom} that for each closed point $g\in G$, we have the induced isomorphism
\begin{gather*}
\uptheta_g^{\cM}\colon\ \cM\simto \sigma_g^*\cM.
\end{gather*}
If $\big(\cN,\uptheta^{\cN}\big)$ is another $G$-equivariant $\cA$-module, we have the $G$-action on $\Hom_{\Qcoh\cA}(\cM,\cN)$ defined by
\begin{gather*}
g\cdot\varphi\defeq\big(\uptheta^{\cN}_g\big)^{-1}\circ\sigma_g^*\varphi\circ\uptheta^{\cM}_g.
\end{gather*}
The following is a generalization of the equality in \eqref{G hom}.
\begin{prop}\label{G hom 2}
Assume that $X$ is of finite type over $k$. We have
\begin{gather*}
\Hom_{\Qcoh_G\cA}\big(\big(\cM,\uptheta^{\cM}\big),\big(\cN,\uptheta^{\cN}\big)\big) =\Hom_{\Qcoh\cA}(\cM,\cN)^G.
\end{gather*}
Moreover, if $G$ is reductive, we have
\begin{gather*}
\Ext_{\Qcoh_G\cA}^i\big(\big(\cM,\uptheta^{\cM}\big),\big(\cN,\uptheta^{\cN}\big)\big) =\Ext_{\Qcoh\cA}^i(\cM,\cN)^G.
\end{gather*}
for any $i\geq0$.
\begin{proof}
The first equality follows from an identical argument as in Remark~\ref{cl pt}. Since $G$ is reductive, the functor
of taking $G$-invariant parts is exact, and thus the second equality follows from the first one.
\end{proof}
\end{prop}

\begin{exa} 
Notation is the same as in Example~\ref{end ex}(1).

\begin{itemize}\itemsep=0pt
\item[$1.$] For any $\cG\in \Qcoh_GX$, the quasi-coherent $\cA$-module \begin{gather*}\cHom_X(\cF,\cG)\end{gather*} has a natural $G$-equivariant structure induced by the $G$-equivariant structure on
 \begin{gather*}
 \cHom_X(\cF,\cG)\in \Qcoh_GX.
 \end{gather*}

\item[$2.$] For $\cM\in \Mod_G\cA$ and $\cF\in \Qcoh_GX$, the tensor product \begin{gather*}\cF\otimes_X \cM\end{gather*} of $\cO_X$-modules is an $\cA$-module with a natural $G$-equivariant structure induced by the $G$-equivariant structure on $\cF\otimes_X \cM\in \Qcoh_GX$.
\end{itemize}
\end{exa}

If $\big(\cA,\uptheta^{\cA}\big)$ is a $G$-equivariant $X$-algebra, the ring homomorphism
$\uprho\colon \cO_X\to \cA$ is $G$-equivariant by Remark~\ref{linearity}. We define the sheaf of rings
 \begin{gather*}
 [\cA/G]
 \end{gather*}
 on the quotient stack $[X/G]$ to be the image of $\cA\in \Qcoh_GX$ by the equivalence
\begin{gather*}
[-/G]\colon\ \Qcoh_GX\simto \Qcoh[X/G]
\end{gather*}
 in \eqref{equiv stack}. Since $[\cO_X/G]$ is canonically isomorphic to the structure sheaf $\cO_{[X/G]}$, we have the following morphism
\begin{gather*}
[\uprho/G]\colon\ \cO_{[X/G]}\to [\cA/G]
\end{gather*}
of sheaf of rings and this makes the sheaf $[\cA/G]$ an algebra over $[X/G]$.
The following is a~generalization of \eqref{equiv stack}, and it follows from Proposition~\ref{main app}.
\begin{prop}\label{quotient correspond}
We have an equivalence
\begin{gather*}
\Qcoh_G\cA\cong \Qcoh [\cA/G]
\end{gather*}
of abelian categories.
\end{prop}

\begin{rem}
If an affine algebraic group $G$ is abelian, by a similar argument as in~\cite[Section~2.1]{bfk2} $G$-equivariant algebras correspond to $\widehat{G}$-graded algebras, where $\widehat{G}$ is the character group of $G$. Since we do not need this correspondence in the present paper, we do not give a~formulation of the correspondence and its proof.
\end{rem}

\subsection{Functors of equivariant modules}
In this subsection, we define fundamental functors between equivariant modules. Let $X$ be a~qua\-si-compact and quasi-separated scheme, and $G$ an algebraic group acting on $X$ by $\sigma\colon G\times X\to X$. Denote by $\pi\colon G\times X\to X$ the natural projection.

\subsubsection{Restrictions and extensions}

Let $\cA$ and $\cB$ be $G$-equivariant $X$-algebras, and let $\varphi\colon \cA\to \cB$ be a $G$-equivariant morphism of~$X$-algebras. The functors in \eqref{restriction} and \eqref{extension} define the {\it restriction} by $\varphi$
\begin{gather*}
(-)_{\varphi}\colon\ \Qcoh_G\cB\to \Qcoh_G\cA
\end{gather*}
and the {\it extension} by $\varphi$
\begin{gather*}
(-)\otimes_{\cA}\cB\colon\ \Qcoh_G\cA\to \Qcoh_G\cB,
\end{gather*}
and we have an adjunction
\begin{equation}\label{rest ext adj}
(-)\otimes_{\cA}\cB\dashv (-)_{\varphi}.
\end{equation}

\subsubsection{Direct image functors and pull-back functors}
Let $Y$ be another quasi-compact and quasi-separated $G$-scheme, and $f\colon X\to Y$ a $G$-equivariant morphism.
Let $\cA$ be a $G$-equivariant $X$-algebra, and $\cB$ a $G$-equivariant $Y$-algebra. Recall from Example~\ref{end ex}(2) that the push-forward $f_*\cA$ is a $G$-equivariant $Y$-algebra and that the pull-back $f^*\cB$ is a $G$-equivariant $X$-algebra. The push-forwards and the pull-backs of $G$-equivariant modules define the following additive functors
\begin{gather*}
f_{\star}\colon\ \Qcoh_G\cA\to \Qcoh_Gf_*\cA,
\\[.5ex]
f^*\colon\ \Qcoh_G\cB \to \Qcoh_Gf^*\cB.
\end{gather*}
Using natural morphisms of equivariant algebras
$\varphi\colon f^{-1}f_*\cA\to \cA$ and $\psi\colon \cB\to f_*f^*\cB$,
we define additive functors
\begin{gather*}
f^{\star}\colon\ \Qcoh_Gf_*\cA\to\Qcoh_G\cA,
\\[.5ex]
f_{*}\colon\ \Qcoh_Gf^*\cB\to \Qcoh_G\cB
\end{gather*}
by $f^{\star}(-)\defeq f^{-1}(-)\otimes_{f^{-1}f_*\cA}\cA$ and $f_{*}\defeq(-)_{\psi}\circ f_{\star}$.
By standard arguments, we see that the above functors induce adjunctions;
\begin{gather*}
f^*\dashv f_{*}\qquad\text{and}\qquad f^{\star}\dashv f_{\star}.
\end{gather*}

 Let $G'$ be another algebraic group acting on another scheme $Z$, and let $\alpha\colon G\to G'$ be a morphism of algebraic groups. Let $h\colon X\to Z$ be a morphism of schemes, and $\cC$ an $G'$-equivariant $Z$-algebra. We say that $h$ is {\it $\alpha$-equivariant} if we have $h(g x)=\alpha(g)h(x)$ for all $(g,x)\in G\times X$. If $h$ is $\alpha$-equivariant, then $h$ induces the pull-back functor
\begin{gather*}
h^*_{\alpha}\colon\ \Qcoh_{G'}Z\to \Qcoh_GX
\end{gather*}
defined by $h^*(\cF,\uptheta)\defeq(h^*\cF, (\alpha\times h)^*\uptheta)$, and we have
 the associated $G$-equivariant $X$-alge\-bra~$h^*\cC$. This functor extends to
the functor
\begin{equation*}
h^*_{\alpha}\colon\ \Qcoh_{G'}\cC\to \Qcoh_Gh^*\cC.
\end{equation*}

\subsubsection{Tensor products and sheaf Homs}

Let $\uprho\colon\cO_{X}\to\cA$ and $\uprho'\colon\cO_{X}\to\cA'$ be $G$-equivariant $X$-algebras. A {\it $G$-equivariant quasi-coherent $($resp.\ coherent$)$ $(\cA',\cA)$-bimodule} is a pair $\big(\cM,\uptheta^{\cM}\big)$ of a $(\cA',\cA)$-bimodule $\cM$ such that the left action of $\cO_{X}$ via $\uprho'$ coincides with the right action of $\cO_{X}$ via $\uprho$ and that $\cM$ is a quasi-coherent (resp.\ coherent) sheaf on $X$ and an isomorphism $\uptheta\colon \pi^*\cM\simto\sigma^*\cM$ of $(\pi^*\cA',\pi^*\cA)$-bimodules such that $\uptheta^{\cM}$ satisfies the condition \eqref{compati}.
 We denote by
$\Qcoh_G(\cA',\cA)$ the category of $G$-equivariant quasi-coherent $(\cA',\cA)$-bimodules, and $\coh_G(\cA',\cA)$ the full subcategory of $G$-equivariant coherent $(\cA',\cA)$-bimodules.

Let $\cF\in \Qcoh_G(\cA',\cA)$, $\cM\in\Qcoh_G\cA'$ and $\cN\in \Qcoh_G\cA$. Then we have the tensor product
\begin{gather*}
\cM\otimes_{\cA'}\cF\in \Qcoh_G\cA.
\end{gather*}
If $\cF\in \coh_G(\cA',\cA)$, we also have the sheaf Hom
\begin{gather*}
\cHom_{\cA}(\cF,\cN)\in \Qcoh_G\cA',
\end{gather*}
and there is a functorial isomorphism
\begin{equation*}
\Hom_{\Qcoh_G\cA}\big(\cM\otimes_{\cA'}\cF,\cN\big)\cong \Hom_{\Qcoh_G\cA'}\big(\cM,\cHom_{\cA}(\cF,\cN)\big).
\end{equation*}
In other words, if $\cF\in \coh_G(\cA',\cA)$, the functor
\begin{gather*}
(-)\otimes_{\cA'}\cF\colon\ \Qcoh_G\cA'\to\Qcoh_G\cA
\end{gather*}
is left adjoint to the functor
\begin{gather*}
\cHom_{\cA}(\cF,-)\colon\ \Qcoh_G\cA\to\Qcoh_G \cA'.
\end{gather*}

The following is a version of projection formula for equivariant $\cA$-modules.

\begin{lem}[Projection formula]\label{proj form}
Let $Y$ be a quasi-compact and quasi-separated scheme with a~$G$-action, and $g\colon Y\to X$ a $G$-equivariant morphism. If $g$ is flat and affine, for any $\cE\in \Qcoh_G Y$ and $\cF\in \Qcoh_G\cA$, we have an isomorphism of $G$-equivariant $\cA$-modules
\begin{gather*}
g_*(\cE)\otimes_{\cO_X} \cF\simto g_*\!\big(\cE\otimes_{\cO_Y}g^*\cF \big).
\end{gather*}
\begin{proof}
Since $\bR g_*\cong g_*$ and $\bL \,g^*\cong g^*$, by the projection formula~\cite[Proposition~3.9.4]{lip}, we~have a quasi-isomorphism
\begin{gather*}
g_*\cE\otimes^{\bL}_{\cO_X}\cF\cong g_*\big(\cE\otimes^{\bL}_{\cO_Y}g^*\cF\big)
\end{gather*}
of complexes of quasi-coherent $\cO_X$-modules. Hence we have isomorphisms
\begin{align*}
g_*(\cE)\otimes_{\cO_X} \cF&\cong H^0\big(g_*(\cE)\otimes^{\bL}_{\cO_X}\cF\big)\\
&\cong H^0\big(g_*\big(\cE\otimes^{\bL}_{\cO_Y}g^*\cF\big)\big)\\
&\cong g_*\big(H^0\big(\cE\otimes^{\bL}_{\cO_Y}g^*\cF\big)\big)\\
&\cong g_*\big(\cE\otimes_{\cO_Y}g^*\cF\big),
\end{align*}
where the third isomorphism follows since $g_*$ is an exact functor. This isomorphism
\begin{gather*}
\varphi\colon\ g_*(\cE)\otimes_{\cO_X} \cF\simto g_*\!\big(\cE\otimes_{\cO_Y}g^*\cF \big)
\end{gather*}
of $\cO_X$-modules is nothing but the composition
\begin{gather*}
 g_*(\cE)\otimes_{\cO_X} \cF\longrightarrow g_*(\cE)\otimes_{\cO_X} g_*g^*\cF\longrightarrow g_*\!\big(\cE\otimes_{\cO_Y}g^*\cF \big),
\end{gather*}
where the first morphism is the morphism induced by the adjunction $\cF\to g_*g^*\cF$ and the second morphism is the canonical morphism defined by $x\otimes y\mapsto x\otimes y$. Since these morphisms are $\cA$-linear and $G$-equivariant, so is the composition $\varphi$. This finishes the proof.
\end{proof}
\end{lem}

\subsubsection{Taking invariant sections.}
Let $H$ be a closed normal subgroup of $G$. Assume that the restriction $\sigma_H\defeq\sigma|_{H\times X}\colon H\times X\to X$ is the trivial action, so that $\sigma_H=\pi_H$, where $\pi_H\defeq\pi|_{H\times X}\colon H\times X\to X$ is the natural projection. Then we have the induced $G/H$-action on $X$ denoted by
\begin{gather*}
\overline{\sigma}\colon\ G/H\times X\to X.
\end{gather*}
We write $\overline{\pi}\colon G/H\times X\to X$ for the natural projection.

For $\big(\cF,\uptheta^{\cF}\big)\in\Qcoh_GX$, we define the subsheaf $\cF^H\subseteq\cF$ to be the kernel of the composition
\[
\begin{tikzcd}
\cF\arrow[r,"\mu"]&(\pi_H)_*(\pi_H)^* \cF\arrow[rr,"\id -(\pi_H)_*\uptheta_H^{\cF}"]&&(\pi_H)_*(\pi_H)^* \cF,
\end{tikzcd}
\]
 where $\mu$ is the adjunction morphism and $\uptheta^{\cF}_H\defeq\uptheta^{\cF}|_{H\times X}\colon (\pi_H)^*\cF\simto (\sigma_H)^*\cF=(\pi_H)^*\cF$.
Then the pair $\big(\cF^H,\uptheta^{\cF}|_{\pi^*\cF^H}\big)$ is a $G$-equivariant quasi-coherent sheaf. Since the restriction of the isomorphism $\uptheta^{\cF}|_{\pi^*\cF^H}\colon {\pi^*\cF^H}\simto {\sigma^*\cF^H}$ to $H\times X$ is the identity, there is a unique $G/H$-equivariant structure $\uptheta^{\cF^H}\colon \overline{\pi}^*\cF^H\simto \overline{\sigma}^*\cF^H$ on $\cF^H$ such that $\uptheta^{\cF}|_{\pi^*\cF^H}={(p\times \id _X)}^*\uptheta^{\cF^H}$, where $p\colon G\to G/H$ is the natural projection. This defines the functor
\begin{equation}\label{inv coh}
(-)^H\colon\ \Qcoh_GX\to \Qcoh_{G/H}X.
\end{equation}

If $\uprho\colon\cO_X\to\cA$ is a $G$-equivariant $X$-algebra, the induced morphism
\begin{gather*}
\cO_X=\cO_X^H\xrightarrow{\uprho^H}\cA^H
\end{gather*}
defines a $G/H$-equivariant $X$-algebra. If $\big(\cM,\uptheta^{\cM}\big)$ is a $G$-equivariant $\cA$-module, the morphism $\alpha\colon \cA\times \cM\to \cM$ defining the $\cA$-module structure on $\cM$ is $G$-equivariant since $\uptheta^{\cM}$ is $\pi^*\cA$-linear. The induced morphism
\begin{gather*}
\alpha^H\colon\ \cA^H\times \cM^H\to \cM^H
\end{gather*}
 defines a $\cA^H$-module structure on $\cM^H$, and so the pair $\big(\cM^H,\uptheta^{\cM^H}\big)$ is a $G/H$-equivariant $\cA^H$-module. Thus the functor \eqref{inv coh} extends to the functor
\begin{gather*}
(-)^H\colon\ \Qcoh_G\cA\to \Qcoh_{G/H}\cA^H.
\end{gather*}
The following is a generalization of~\cite[Lemma~2.22]{bfk}.

\begin{prop}
The composition
\[
\begin{tikzcd}
\Qcoh_{G/H}\cA^H\arrow[r,"\id_{p}^*"]&\Qcoh_{G}\cA^H\arrow[rr,"(-)\otimes_{\cA^H}\cA"]&&\Qcoh_{G}\cA
\end{tikzcd}
\]
is left adjoint to the functor $(-)^H$.
\begin{proof}
Let $\cM\in \Qcoh_{G/H}\cA^H$ and $\cN\in \Qcoh_G\cA$. The result follows from the following sequences of isomorphisms
\begin{align*}
\Hom_{\Qcoh_{G}\cA}\big(\id_p^*\cM\otimes_{\cA^H}\cA,\cN\big)
&\cong\Hom_{\Qcoh_{G}\cA^H}\big(\id_p^*\cM,\cN_{\iota}\big)
\\
&\cong \Hom_{\Qcoh_{G/H}\cA^H}\big(\cM,(\cN_{\iota})^H\big)
\\
&\cong \Hom_{\Qcoh_{G/H}\cA^H}\big(\cM,\cN^H\big),
\end{align*}
\looseness=1
where $\cN_{\iota}$ is the restriction of $\cN$ by the natural inclusion $\iota\colon \cA^H\hookto \cA$, the first isomorphism follows from \eqref{rest ext adj}, the second isomorphism follows from an identical argument as in the proof of~\cite[Lemma~2.22]{bfk}, and the last isomorphism follows from a natural isomorphism \mbox{$(\cN_{\iota})^H\cong \cN^H$}.
\end{proof}
\end{prop}

\subsubsection{Restriction functors and induction functors}
 Let $H$ be a closed subgroup of $G$, and
 \begin{gather*}
 \alpha\colon\ H\hookto G
 \end{gather*}
 the natural inclusion morphism. Let $\cA$ be a $G$-equivariant $X$-algebra. We define the {\it restriction functor}
\begin{gather*}
\Res^G_H\defeq (\id_X)_{\alpha}^*\colon\ \Qcoh_G\cA\to \Qcoh_H\cA
\end{gather*}
to be the pull-back by the identity morphism $\id_X\colon X\to X$ that is $\alpha$-equivariant. If $H$ is trivial, we write $\Res\defeq\Res^G_{\{1\}}$, which is nothing but the forgetful functor $\big(\cM,\uptheta^{\cM}\big)\mapsto \cM$.

Next, we will construct the adjoint functor of this restriction functor. We define an $H$-action on $G\times X$ by
\begin{gather*}
h\cdot(g,x)\defeq\big(g\big(h^{-1}\big), h x\big)
\end{gather*}
for any $h\in H$, $g\in G$ and $x\in X$, and we write
\begin{gather*}
G\timesH X\defeq[G\times X/H]
\end{gather*}
for the associated quotient stack. By~\cite[Lemma~2.16(a)]{bfk}, the quotient stack $G\timesH X$ is representable by the scheme $G/H\times X$.
Since the morphism $\sigma\colon G\times X\to X$ defining the $G$-action on $X$ is $H$-invariant, we have the induced morphism
\begin{gather*}
\sigma^H\colon\ G\timesH X\to X.
\end{gather*}
We define the morphism
\begin{gather*}
\varepsilon_X^H\colon\ X\to G\timesH X
\end{gather*}
to be the composition $X\xrightarrow{\varepsilon_X}G\times X\xrightarrow{\,\,q\,\,}G\timesH X$,
where $q$ is the canonical quotient map. Then we have $\sigma^{H}\circ\varepsilon_X^H=\id_X$. The $G$-action on $G\times X$ given by
\begin{gather}\label{3b}
\widetilde{\sigma}\colon\ G\times (G\times X)\to G\times X,\qquad
(g,g',x)\mapsto (gg',x)
\end{gather}
induces the $G$-action on $G\timesH X$. With respect to this $G$-action on $G\timesH X$, the morphism $\varepsilon_X^H$ is~$\alpha$-equivariant, and $\sigma^H$ is $G$-equivariant. Thus $\big(\sigma^H\big)^*\cA$ is a $G$-equivariant $G\timesH X$-algebra, and we have the following functor
\begin{gather*}
\Phi\defeq\big(\varepsilon_X^H\big)^*_{\alpha}\colon\ \Qcoh _G\big(\sigma^H\big)^*\cA\to \Qcoh_H\cA.
\end{gather*}
The following is a generalization of~\cite[Lemma~2.13]{bfk}

\begin{lem}
The functor $\Phi\colon \Qcoh _G\big(\sigma^H\big)^*\cA\to \Qcoh_H\cA$ is an equivalence.
\end{lem}
\begin{proof} This can be proved by a similar argument as in~\cite[Lemma~1.3]{tho}.
 Since $\big(\sigma^H\big)^*\cA$ is isomorphic to the restriction of the big fppf sheaf $[\sigma^*\cA/H]$ to the small Zariski site of~\mbox{$G\timesH X$}, by Proposition~\ref{quotient correspond} we have an equivalence $\Qcoh\big(\sigma^H\big)^*\cA\cong \Qcoh_H\sigma^*\cA$. Thus objects in $\Qcoh _G\big(\sigma^H\big)^*\cA$ are identified with pairs
\begin{equation}\label{object 1}
\Big(\big(\cM,\uptheta^{\cM}\big), \uptheta^{\left(\cM,\uptheta^{\cM}\right)}\Big)
\end{equation}
{\sloppy of $\big(\cM,\uptheta^{\cM}\big)\in \Qcoh_H\sigma^*\cA$ and an isomorphism $\uptheta^{\left(\cM,\uptheta^{\cM}\right)}\colon \widetilde{\pi}^*\big(\cM,\uptheta^{\cM}\big)\simto \widetilde{\sigma}^*\big(\cM,\uptheta^{\cM}\big)$ in $\Qcoh_H(\sigma\circ\widetilde{\sigma})^*\cA$ satisfying the conditions as in \eqref{compati}, where $\widetilde{\pi}\colon G\times G\times X\to G\times X$ and $\widetilde{\sigma}\colon G\times G\times X\to G\times X$ are the projection and the group action in \eqref{3b} respectively. If we write $\widetilde{\uptheta}^{\cM}\colon \widetilde{\pi}^*\cM\simto\widetilde{\sigma}^*\cM$ for the isomorphism in $\Qcoh(\sigma\circ\widetilde{\sigma})^*\cA$ defining the isomorphism $\uptheta^{\left(\cM,\uptheta^{\cM}\right)}$, then the pair $(\cM,\widetilde{\uptheta}^{\cM})$ is an object in $\Qcoh_G\sigma^*\cA$, and $\uptheta^{\cM}$ defines a $H$-equivariant structure on~$\big(\cM,\widetilde{\uptheta}^{\cM}\big)$. Thus the object~\eqref{object 1} uniquely corresponds to an object in $\Qcoh_H[\sigma^*\cA/G]$. Since $G$ trivially acts on $X$ in the $G$-action~\eqref{3b}, the composition of $\varepsilon_X\colon X\to G\times X$ and the natural projection $G\times X\to [G\times X/G]$ gives an isomorphism $ X\simto[G\times X/G]$. Via this isomorphism, the sheaf $[\sigma^*\cA/G]$ on $[G\times X/G]$ corresponds to the sheaf $\cA$ on $X$. Hence the assignment
 \begin{gather*}
 \Bigl(\big(\cM,\uptheta^{\cM}\big), \uptheta^{\left(\cM,\uptheta^{\cM}\right)}\Bigr)\mapsto \big(\varepsilon_X^*\cM,\uptheta^{\cM}|_{H\times \{1\}\times X}\big)
 \end{gather*}}\noindent
 defines an equivalence $\Qcoh_G[\sigma^*\cA/H]\simto \Qcoh_H\cA$, and $\Phi$ is isomorphic to this equivalence via $\Qcoh _G\big(\sigma^H\big)^*\cA\cong \Qcoh_G[\sigma^*\cA/H]$.
 \end{proof}

Since the morphism $\sigma^H\colon G\timesH X\to X$ is $G$-equivariant, we have the direct image functor $\sigma^H_*\colon\Qcoh_G\big(\sigma^{H}\big)^*\cA\to\Qcoh_G\cA$, and
we define the {\it induction functor}
\begin{gather*}
\Ind^G_H\colon\ \Qcoh_H \cA\to \Qcoh_G\cA
\end{gather*}
to be the composition $\big(\sigma^H\big)_*\circ \Phi^{-1}$. If $H$ is trivial, we write $\Ind\defeq\Ind^G_{\{1\}}$.
\begin{lem}\label{split mono} Notation is the same as above.
\begin{itemize}\itemsep=0pt
\item[$1.$] The functor $\Res^G_H$ is exact, and if $G/H$ is affine, the functor $\Ind^G_H$ is also exact.
\item[$2.$] The restriction functor $\Res^G_H$ is left adjoint to the induction functor $\Ind^G_H$.
\item[$3.$] If $H$ is a normal subgroup of $G$ and $G/H$ is reductive affine algebraic group, the adjunction morphism
\begin{gather*}
\mu\colon\ \id\to \Ind^G_H\circ \Res^G_H
\end{gather*}
is a split mono, i.e., there exists a functor morphism $\nu\colon\Ind^G_H\circ\Res^G_H\to \id$ such that the composition $\nu\circ\mu$ is the identity morphism of functors.

\end{itemize}
\begin{proof}
(1) Since the morphism $\sigma^H$ is flat, the pull-back $\big(\sigma^H\big)^*$ is exact. Since $\sigma^{H}\circ\varepsilon_X^H=\id_X$, the functor $\Res^G_H$ is isomorphic to the composition $\Phi\circ\big(\sigma^{H}\big)^*$, and thus $\Res^G_H$ is exact. If $G/H$ is affine, the morphism $\sigma^H$ is an affine morphism since $G\timesH X$ is isomorphic to $G/H\times X$. Then the direct image $\big(\sigma^{H}\big)_*$ is exact, and so is $\Ind^G_H$.

(2) This follows from the adjunction $\big(\sigma^{H}\big)^*\dashv\big(\sigma^{H}\big)_*$ and the isomorphism $\Res^G_H\cong \Phi\circ\big(\sigma^{H}\big)^*$.

(3) It is enough to show that the adjunction morphism
\begin{gather*}
\widetilde{\mu}\colon\ \id \to \big(\sigma^H\big)_*\big(\sigma^H\big)^*
\end{gather*}
is a split mono.
 Note that we have the following cartesian square
\[
\begin{tikzcd}
G\timesH X\arrow[rr,"\sigma^H"]\arrow[d]&&X\arrow[d, "q"]
\\
G/H\arrow[rr, "p"]&&\Spec k,
\end{tikzcd}
\]
where $p$ and $q$ are the morphisms defining the base field $k$.
If $G/H$ is a reductive affine algebraic group, it is linearly reductive. Then the natural morphism
$k\to p_*p^*k=k[G/H]$
is a split mono, and so is the adjunction morphism
\begin{gather*}
\cO_X\to\big(\sigma^H\big)_*\big(\sigma^H\big)^*\cO_{X}
\end{gather*}
by the base change formula for the above cartesian square. For any $\cM\in \Qcoh_G\cA$, by~Lem\-ma~\ref{proj form} the adjunction $\widetilde{\mu}(\cM)\colon \cM\to \big(\sigma^H\big)_*\big(\sigma^H\big)^*\cM$ is isomorphic to the tensor product \begin{gather*}\Bigr(\cO_X\to\big(\sigma^H\big)_*\big(\sigma^H\big)^*\cO_{X}\Bigr)\otimes_{\cO_X}\cM,\end{gather*} and therefore $\widetilde{\mu}(\cM)$ is also a split mono.
\end{proof}
\end{lem}

We will apply the above result to the following:

\begin{lem}\label{faithful general}
Let $\scrC$ and $\scrD$ be additive categories, and
$F\colon \scrC\to \scrD$ and $G\colon \scrD\to \scrC$ additive functors. If $F$ is left adjoint to $G$, and if the adjunction morphism $\mu\colon \id_{\scrC}\to G\circ F$ is a split mono, then the functor $F$ is faithful.
\begin{proof}
For a morphism $f\colon A\to B$ in $\scrC$, assume that $F(f)=0$. It is enough to show that $f=0$. We have a commutative diagram
\[
\begin{tikzcd}
A\arrow[rr, "\mu(A)"]\arrow[rrrr, bend left, "\id"]\arrow[d, "f"']&&GF(A)\arrow[d, "GF(f)"]\arrow[rr, "\nu(A)"]&&A\arrow[d, "f"]\\
B\arrow[rr, "\mu(B)"]&&GF(B)\arrow[rr, "\nu(B)"]&& B.
\end{tikzcd}
\]
Therefore, $f=\nu(B)\circ GF(f)\circ \mu(A)=\nu(B)\circ G(0)\circ \mu(A)=0$.
\end{proof}
\end{lem}

\subsection{Fundamental properties of equivariant modules} In this subsection, we prove that the category of equivariant modules over a certain equivariant algebra has enough injectives and enough locally free modules.

Let $X$ be a quasi-compact and quasi-separated scheme with an action from an affine algebraic group $G$, and $\uprho\colon \cO_X\to\cA$ a $G$-equivariant $X$-algebra.

\begin{prop}\label{grothen prop}
 The category $\Qcoh_G\cA$ is a Grothendieck category. In particular, it has enough injectives.
\begin{proof}
The category $\Qcoh_G\cA$ is an abelian category with small direct sums, and so it suffices to prove that (1) filtered colimits are exact in $\Qcoh_G\cA$ and that (2) $\Qcoh_G\cA$ has a generator.

(1) For a point $x\in X$, denote by $F_x\colon \Qcoh_G\cA\to \Mod\cA_x$ the composition
\begin{gather*}
\Qcoh_G\cA\xrightarrow{\Res}\Qcoh\cA\xrightarrow{(-)_x}\Mod\cA_x,
\end{gather*}
where the latter is taking the stalk at $x$, and $\cA_x$ is the stalk of $\cA$ at $x$. Then a sequence in~$\Qcoh_G\cA$ is exact if and only if for every $x\in X$ the sequence in $\Mod\cA_x$ induced by $F_x$ is exact. Therefore, since $F_x$ commutes with filtered colimits and $\Mod \cA_x$ is a Grothendieck category, filtered colimits in $\Qcoh_G\cA$ are also exact.

(2)
It is well known that $\Qcoh_GX$ is a Grothendieck category (more generally, the category of quasi-coherent sheaves on an algebraic stack is a Grothendieck category~\cite[Proposition~14.2; 0781]{stacks}, \cite[Corollary~5.10]{tlrv}). In particular, $\Qcoh_GX$ has a generator $\cG\in \Qcoh_G X$. We show that
\begin{gather*}\cG_{\cA}\defeq\cG\otimes_{\cO_X}\cA\in \Qcoh_G \cA\end{gather*} is a generator. Let $\cM\in\Qcoh_G\cA$ be an object. Then, there is a set $I$ and a surjective morphism
\begin{gather*}
p\colon\ \cG^{\oplus I}\surjto \cM_{\uprho}
\end{gather*}
in $\Qcoh_G X$. Since the functor $(-)\otimes_{\cO_X}\cA\colon \Qcoh_G X\to \Qcoh_G\cA$ is left adjoint to the functor $(-)_{\uprho}$, it is right exact and commutes with small direct sums. Hence the extension of $p$ by $\uprho\colon \cO_X\to \cA$ defines a surjective morphism
\begin{gather*}
p_{\cA}\colon\ \cG_{\cA}^{\oplus I}\surjto \cM_{\uprho}\otimes_{\cO_X}\cA.
\end{gather*}
Composing this with a natural surjection $\cM_{\uprho}\otimes_{\cO_X}\cA\to \cM$, we have a surjective morphism $\cG_{\cA}^{\oplus I}\surjto \cM$. This shows that $\cG_{\cA}$ is a generator in $\Qcoh_G \cA$.
\end{proof}
\end{prop}

{\samepage\begin{dfn} Notation is the same as above.
\begin{itemize}\itemsep=0pt
\item[1.] A $G$-equivariant $\cA$-module $\big(\cM,\uptheta^{\cM}\big)$ is said to be {\it locally free}, if there is an open covering $\{U_i\hookto X\}_{i\in I}$ of $X$ such that for every $U_i$ the restriction $\cM|_{U_i}$ of the underlying $\cA$-module~$\cM$ is a free $\cA|_{U_i}$-module.

\item[2.] We say that $\big(\cA,\uptheta^{\cA}\big)$ {\it satisfies the resolution property}, if for any $G$-equivariant coherent $\cA$-module $\cN\in \coh_G\cA$, there exists a surjective morphism $\big(\cE,\uptheta^{\cE}\big)\surjto \big(\cN,\uptheta^{\cN}\big)$ in $\coh_G\cA$ from a $G$-equivariant locally free coherent $\cA$-module $\big(\cE,\uptheta^{\cE}\big)$.
\item[3.] A $G$-scheme $X$ {\it satisfies the $G$-equivariant resolution property} if the $G$-equivariant $X$-al\-ge\-bra $(\cO_X,\uptheta^{\can})$ has the resolution property.
\end{itemize}
\end{dfn}

}

\begin{rem}
(1) If $X=\Spec R$ is an affine noetherian scheme, a $G$-equivariant coherent $R$-algebra $A$ satisfies the resolution property if and only if the category $\Mod_GA$ has enough projectives.

(2) Assume that $X$ is noetherian and that $\cA$ is coherent. If $\big(\cA,\uptheta^{\cA}\big)$ satisfies the resolution property, for every $G$-equivariant $\cA$-module $\cM$, there is a surjection $\cE\surjto \cM$ from a $G$-equivariant locally free $\cA$-module $\cE$. This follows since there is a set $\{\cM_i\}_{i\in I}$ of $G$-equivariant coherent submodules $\cM_i$ of $\cM$ such that the natural map $\bigoplus_{i\in I}\cM_i\to \cM$ is surjective.
\end{rem}

\begin{prop}\label{resol prop}
Assume that $\big(\cA,\uptheta^{\cA}\big)$ is coherent and that $X$ has the $G$-equivariant resolution property. Then $\big(\cA,\uptheta^{\cA}\big)$ satisfies the resolution property.
\begin{proof}
Let $\cM\in \coh_G\cA$ be a $G$-equivariant coherent $\cA$-module. By the assumption, there is a surjective morphism $p\colon \cE\surjto \cM_{\uprho}$ from a $G$-equivariant locally free coherent $\cO_X$-module $\cE$. Then the extension
\begin{gather*}
p\otimes_{\cO_X}\cA\colon\ \cE\otimes _{\cO_X}\cA\to \cM_{\uprho}\otimes_{\cO_X}\cA
\end{gather*}
of $p$ by $\uprho\colon \cO_{X}\to \cA$ is also a surjection in $\Qcoh_G \cA$. Composing with the natural surjection $\cM_{\uprho}\otimes_{\cO_{X}} \cA\surjto \cM$ gives a surjective morphism $\cE\otimes_{\cO_X} \cA\surjto \cM$ from a $G$-equivariant locally free $\cA$-module.
\end{proof}
\end{prop}

\begin{dfn}
A $G$-equivariant (resp.\ coherent) $\cA$-module $\big(\cM,\uptheta^{\cM}\big)$ has a {\it finite locally free resolution} in $\Qcoh_G\cA$ (resp.\ in $\coh_G\cA$) if there exists an exact sequence
\begin{gather*}
0\to \cE^{-n}\to\cdots\to\cE^0\to \cM\to0
\end{gather*}
in $\Qcoh_G\cA$ (resp.\ in $\coh_G\cA$) such that each $\cE^i$ is locally free.
\end{dfn}

\begin{lem}\label{enough proj}
Assume that $X$ is a normal scheme of finite type over $k$ and that $G$ is affine.
If~$X$ has an ample family of line bundles, then $\big(\cA,\uptheta^{\cA}\big)$ satisfies the resolution property.
\begin{proof}
This follows from~\cite[Theorem~2.29]{bfk} and Proposition~\ref{resol prop}.
\end{proof}
\end{lem}

\begin{lem}\label{finite proj resol}
Let $R$ be a normal ring of finite type over $k$, and $G$ an affine reductive algebraic group acting on $\Spec R$. Let $\big(A,\uptheta^{A}\big)$ be a $G$-equivariant coherent $R$-algebra, and $\big(M,\uptheta^M\big)\in \Mod_GA$ a $G$-equivariant $A$-module. If the underlying $A$-module $M\in \Mod A$ has a finite projective resolution, so does $\big(M,\uptheta^M\big)$. If $M$ is finitely generated and has a finite projective resolution in $\fmod A$, then $\big(M,\uptheta^M\big)$ has a finite projective resolution in $\fmod_GA$.
\begin{proof}
By Lemma~\ref{enough proj}, the categories $\Mod_GA$ and $\fmod_GA$ have enough projectives. Since~$A$ is coherent $R$-algebra, $\fmod_GA$ is an abelian subcategory of $\Mod_GA$. It is standard that an~object~$x$ in an abelian category $\scrA$ has a finite projective resolution if and only if there exists $n>0$ such that $\Ext_{\scrA}^i(x,y)=0$ for any $i>n$ and any $y\in \scrA$. Thus the statements follow from Proposition~\ref{G hom 2}.
\end{proof}
\end{lem}

\subsection{Equivariant tilting modules and derived equivalences}\label{derived section}

In this subsection, we prove that an equivariant tilting module induces a derived equivalence of~equivariant algebras, which is considered as an equivariant Morita theory.

Let $X$ be a separated scheme of finite type over $k$, and $G$ an affine algebraic group acting on~$X$.
Let $H\subseteq G$ be a closed normal subgroup of $G$, and $\cA$ a $G$-equivariant coherent $X$-algebra.

\begin{dfn}
Let $\cT\in \coh_G\cA$ be a $G$-equivariant coherent $\cA$-module.
\begin{itemize}\itemsep=0pt
\item[1.] $\cT$ is called a {\it $(G,H)$-tilting module} (resp.\ {\it partial $(G,H)$-tilting module}), if the restriction $\Res^G_H(\cT)$ is a tilting object (resp.\ partial tilting object) in $\Qcoh_H\cA$.
\item[2.] $\cT$ is called a {\it $G$-tilting module} (resp.\ {\it partial $G$-tilting module}) if it is $(G,\{1\})$-tilting (resp.\ partial $(G,\{1\})$-tilting).
\end{itemize}
\end{dfn}
Let $R$ be a normal ring of finite type over $k$, and suppose that there is a $G$-action $\sigma\colon G\times \Spec R\to \Spec R$ on $\Spec R$ such that the $H$-action on $\Spec R$, which is given by the restriction of $\sigma$, is trivial. Let $f\colon X\to \Spec R$ be a $G$-equivariant morphism such that $f_*\cO_X=R$ and the associated morphism $\overline{f}\colon [X/H]\to \Spec R$ is proper. Let $\big(\cT,\uptheta^{\cT}\big)\in \coh_G\cA$ be
a $G$-equivariant coherent $\cA$-module. Then the endomorphism ring
\begin{gather*}
\Lambda\defeq\End_{\cA}(\cT)=f_*\cEnd_{\cA}(\cT)
\end{gather*}
of the underlying coherent $\cA$-module $\cT$ is a $G$-equivariant $R$-algebra.
We define the functor
\begin{gather*}
\F\colon\ \Qcoh_G\cA\to \Mod_{G/H}\Lambda^H
\end{gather*}
to be the composition
\[
\begin{tikzcd}
\Qcoh_G\cA\arrow[rr,"\cHom_{\cA}(\cT\mbox{\hspace{-0.5mm}\tiny ,}-)"]&&\Qcoh_G\cEnd_{\cA}(\cT)\arrow[r,"f_{\star}"]&\Mod_G\Lambda\arrow[r,"(-)^H"]&\Mod_{G/H}\Lambda^H.
\end{tikzcd}
\]
Then the composition
\begin{gather*}
\G\defeq\left(-\otimes \cT\right)\circ f^{\star}\circ \left(\id_p^*(-)\otimes_{\cA^H}\cA\right) \colon\ \Mod_{G/H}\Lambda^H\to \Qcoh_G\cA
\end{gather*}
is left adjoint to $\F$, and it preserves equivariant coherent modules. Note that since the morphism $\overline{f}\colon [X/H]\to \Spec R$ is proper, the functor $\F$ also preserves equivariant coherent modules. The following is one of our motivations of considering equivariant algebras and equivariant modules.

\begin{thm}\label{derived equiv}
Assume that $G/H$ is reductive and that every $\Lambda^H$-module $M\in \Mod\Lambda^H$ is of finite projective dimension.
\begin{itemize}\itemsep=0pt
\item[$1.$]
If $\big(\cT,\uptheta^{\cT}\big)$ is partial $(G,H)$-tilting, then the functor
\begin{gather*}
\bL \G\colon\ \Db\big(\Mod_{G/H}\Lambda^H\big)\hookto \Db\big(\Qcoh_G{\cA}\big)
\end{gather*}
is fully faithful, and it restricts to the fully faithful functor
\begin{gather*}
\bL \G\colon\ \Db\big(\fmod_{G/H}\Lambda^H\big)\hookto \Db\big(\coh_G{\cA}\big).
\end{gather*}
\item[$2.$] If $\big(\cT,\uptheta^{\cT}\big)$ is $(G,H)$-tilting, then the functor
\begin{gather*}
\bL \G\colon\ \Db\big(\Mod_{G/H}\Lambda^H\big)\simto \Db\big(\Qcoh_G{\cA}\big)
\end{gather*}
is an equivalence, and it restricts to the equivalence
\begin{gather*}
\bL \G\colon\ \Db\big(\fmod_{G/H}\Lambda^H\big)\simto \Db\big(\coh_G{\cA}\big).
\end{gather*}
\end{itemize}
\end{thm}

Before we prove this theorem, we prepare the following lemmas.
\begin{lem}\label{faithful lemma}
Notation is the same as above, and assume that $G/H$ is reductive.
Then the restrictions
\begin{gather*}
\Res^G_H\colon\ \Db\big(\Qcoh_G\cA\big)\to \Db\big(\Qcoh_H\cA\big),
\\[.5ex]
\Res\colon\ \Db\big(\Mod_{G/H}\Lambda^H\big)\to \Db\big(\Mod\Lambda^H\big)
\end{gather*}
are faithful functors.
\begin{proof}
We only prove $\Res^G_H$ is faithful. Since $G/H$ is affine, $\Ind^G_H\colon \Qcoh_H\cA\to \Qcoh_G\cA$ is an~exact functor, and so it extends to the functor
\begin{gather*}
\Ind^G_H\colon\ \Db\big(\Qcoh_H\cA\big)\to \Db\big(\Qcoh_G\cA\big)
\end{gather*}
which is right adjoint to $\Res^G_H$.
 Since $G/H$ is reductive, the adjunction $\mu\colon \id_{\Qcoh_G\cA} \to\Ind^G_H\Res^G_H$ is a split mono by Lemma~\ref{split mono}, and this splitting naturally extends to the splitting of the adjunction $\mu\colon \id_{\Db\left(\Qcoh_G\cA\right)} \to\Ind^G_H\Res^G_H$ since $\Ind^G_H$ and $\Res^G_H$ are exact functors. Hence the functor $\Res^G_H$ is faithful by Lemma~\ref{faithful general}.
\end{proof}
\end{lem}

\begin{lem}\label{reduction lemma}
For each $i=1,2$, let $F_i\colon \scrC_i\to \scrD_i$ and $G_i\colon \scrD_i\to \scrC_i$ be exact functors between triangulated categories such that $G_i\dashv F_i$. Denote by $\sigma_i\colon G_i\circ F_i\to \id$ and $\tau_i\colon \id\to F_i\circ G_i$ the adjunction morphisms, and let $P\colon \scrC_1\to \scrC_2$ and $Q\colon \scrD_1\to\scrD_2$ be faithful exact functors. Assume that we have functor isomorphisms $\varphi\colon Q\circ F_1\simto F_2\circ P$ and $\psi\colon P\circ G_1\simto G_2\circ Q$ and that the following diagrams are commutative
\[
\begin{tikzcd}
PG_1F_1\arrow[rr, "\sim"]\arrow[rd, "P\sigma_1"']&&G_2F_2P\arrow[ld, "\sigma_2P"]&&
QF_1G_1\arrow[rr, "\sim"]&&F_2G_2Q\\
&P&&&&Q\arrow[lu, "Q\tau_1"]\arrow[ru, "\tau_2Q"']&
\end{tikzcd}
\]
where the top horizontal arrows are the isomorphisms induced by $\varphi$ and $\psi$.
\begin{itemize}\itemsep=0pt
\item[$1.$] If $F_2$ $($resp.\ $G_2)$ is fully faithful, so is $F_1$ $($resp.\ $G_1)$.
\item[$2.$] If $F_2$ $($resp.\ $G_2)$ is an equivalence, so is $F_1$ $($resp.\ $G_1)$.
\end{itemize}
\begin{proof}
(2) follows from (1). We only prove (1) for $F_i$, since the statement for $G_i$ follows by a~simi\-lar argument. For an object $C\in \scrC_1$ consider the following triangle
\begin{gather*}
 G_1F_1(C)\xrightarrow{\sigma_1(C)} C\to{\rm Cone}(\sigma_1(C))\to G_1F_1(C)[1].
\end{gather*}
It suffices to prove that ${\rm Cone}(\tau_1(C))$ is the zero object. By the assumption, $P({\rm Cone}(\sigma_1(C)))$ is isomorphic to the cone of the morphism $\sigma_2(P(C))$, which is the zero object since $F_2$ is fully faithful. Hence ${\rm Cone}(\sigma_1(C))$ is also the zero object since $P$ is faithful.
\end{proof}
\end{lem}

\noindent
{\bf Proof of Theorem~\ref{derived equiv}.}
We prove (1) and (2) simultaneously. We define a functor
\begin{gather*}
\F_H\colon\ \Qcoh_H\cA\to \Mod\Lambda^H
\end{gather*}
to be the composition
\[
\begin{tikzcd}
\Qcoh_H\cA\arrow[rr,"\cHom_{\cA}(\cT\mbox{\hspace{-0.5mm}\tiny ,}-)"]&&\Qcoh_H\cEnd_{\cA}(\cT)\arrow[r,"f_{\star}"]&\Mod_H\Lambda\arrow[r,"(-)^H"]&\Mod\Lambda^H.
\end{tikzcd}
\]
If we set $\cT_H\defeq\Res^G_H\big(\cT,\uptheta^{\cT}\big)\in \Qcoh_H\cA$, since $\Lambda^H\cong \End_{\cA}(\cT)^H=\End_{\Qcoh_H\cA}(\cT_H)$, the functor $\F_H$ is isomorphic to the functor
\begin{gather*}
\Hom_{\Qcoh_H\cA}(\cT_H,-)\colon\ \Qcoh_H\cA\to \Mod\Lambda^H.
\end{gather*}
Note that $\F_H$ has a left adjoint functor $\G_H\colon \Mod\Lambda^H\to \Qcoh_H\cA$ which is given by
\begin{gather*}
\G_H(-)\defeq f^{\star}\big(\id_p^*(-)\otimes_{\Lambda^H}\Lambda\big)\otimes_{\cEnd(\cT)}\cT\cong f^{-1}\big(\id_p^*(-)\otimes_{\Lambda^H}\Lambda\big)\otimes_{f^{-1}\Lambda}\cT,
\end{gather*}
where $p\colon H\to H/H=\{1\}$ is the natural projection. By construction, we have natural isomorphisms $\F_H(\cT_H)\cong \Lambda^H$ and $\G_H\big(\Lambda^H\big)\cong \cT_H$, and this implies that the adjunction morphisms
\begin{gather*}
\Lambda^H\to \F_H\G_H\big(\Lambda^H\big) \qquad\text{and}\qquad \G_H\F_H(\cT_H)\to \cT_H
\end{gather*}
are isomorphisms.
Since $\cT_H$ is a partial tilting object in $\Qcoh_H\cA$, by Theorem~\ref{general tilting}, the functor
\begin{gather*}
\bL\G_H\colon\ \Db\big(\Mod\Lambda^H\big)\to \Db\big(\Qcoh_H\cA\big)
\end{gather*}
is fully faithful, and if $\cT_H$ is tilting, it is an equivalence.
Now we have the following commutative diagram:
\[
\begin{tikzcd}
\Db\big(\Mod_{G/H}\Lambda^H\big)\arrow[rr,"\bL\G"]\arrow[d, "\Res"']&&\Db\big(\Qcoh_G\cA\big)\arrow[d, "\Res^G_H"]\arrow[rr,"\bR\F"]&&\Db\big(\Mod_{G/H}\Lambda^H\big)\arrow[d, "\Res"']\\
\Db\big(\Mod\Lambda^H\big)\arrow[rr, "\bL \G_H"]&&\Db\big(\Qcoh_H\cA\big)\arrow[rr, "\bR \F_H"]&&\Db\big(\Mod\Lambda^H\big).
\end{tikzcd}
\]
Hence (1) and (2) follows from Lemmas~\ref{faithful lemma} and~\ref{reduction lemma}.\qed

 The following is a special case of Theorem~\ref{derived equiv}, which is an equivariant version of derived equivalences induced by tilting bundles and tilting modules.

 \begin{cor} Notation is the same as above. Assume that $G$ is reductive.
 \begin{itemize}\itemsep=0pt
 \item[$1.$] Let $\cT$ be a $G$-equivariant vector bundle on $X$ and set $\Lambda\defeq\End_X(\cT)$. If $\cT$ is $G$-tilting and $\Lambda$ is of finite global dimension, we have an equivalence
 \begin{gather*}
 \Db\big(\coh_GX\big)\simto \Db\big(\fmod_G\Lambda\big).
 \end{gather*}
 \item[$2.$] Let $\Gamma$ be a $G$-equivariant $R$-algebra, and $T$ a $G$-equivariant $\Gamma$-module. Set $\Lambda\defeq\End_{\Gamma}(T)$. If $T$ is $G$-tilting and $\Lambda$ is of finite global dimension, we have an equivalence
 \begin{gather*}
 \Db\big(\fmod_G\Gamma\big)\simto \Db\big(\fmod_G\Lambda\big).
 \end{gather*}
 \end{itemize}
\end{cor}

\section{Equivariant tilting objects and factorizations}

In this section, we prove that equivariant tilting modules induce equivalences of derived factorization categories.

\subsection{Tilting objects and factorizations}\label{general morita}
In this subsection, we prove that a tilting object in a Gorthendieck category induces an equivalence of derived factorization categories. To prove this, we need the following lemmas:

Let $\scrA$ be a Grothendieck category and $(\scrE,\Phi,w)$ a triple as in \eqref{triple}.
\begin{lem}\label{qis}
Assume that we have $\Ext^i_{\scrA}(A,B)=0$ for arbitrary objects $A,B\in \scrE$ and all $i>0$. Then for objects $E,F\in \Fact(\scrE,\Phi,w)$ and an injective resolution $\varepsilon\colon F\to J^{\bullet}$ in $\Ch(\Fact(\scrA,\Phi,w))$, the map
\begin{gather*}
\varepsilon_*\colon\ \Hom_{\dgFact(\scrA,\Phi,w)}(E,F)^p\to\Hom_{\dgFact(\scrA,\Phi,w)}(E,J^{\bullet})^p
\end{gather*}
of cochain complexes of abelian groups is a quasi-isomorphism for every $p\in \bZ$. Here injective resolution means that $\varepsilon$ is a quasi-isomorphism given by an injection $\varepsilon_0\colon F\hookto J^0$, and all components of the factorizations $J^n$ are injective objects in $\scrA$.

\begin{proof}
Since $\Hom_{\dgFact(\scrA,\Phi,w)}(E,F)^p\cong \Hom_{\dgFact(\scrA,\Phi,w)}(E[p],F)^0$, we may assume that $p=0$. If~$\varepsilon\colon F\to J^{\bullet}$ is an injective resolution, then we have the induced injective resolutions $\varepsilon_i\colon F_i\to J^{\bullet}_i$ for $i=0,1$.
Then the map $\varepsilon_*\colon \Hom_{\dgFact(\scrA,\Phi,w)}(E,F)^0\to\Hom_{\dgFact(\scrA,\Phi,w)}(E,J^{\bullet})^0$ of cochain complexes is the direct sum of two maps
\begin{gather*}
(\varepsilon_*)_i\colon\ \Hom_{\scrA}(E_i,F_i)\to \Hom_{\scrA}(E_i,J^{\bullet}_i)\qquad (i=0,1)
\end{gather*}
of cochain complexes. The $i$-th cohomology of the cochain complex $\Hom_{\scrA}(E_i,J^{\bullet}_i)$ is nothing but $\Ext^i_{\scrA}(E_i,F_i)$. Hence, by the assumption, the map $(\varepsilon_*)_i$ is a quasi-isomorphism, and so is~$\varepsilon_*$.
\end{proof}
\end{lem}

\begin{lem}\label{add ff}
Assume that we have $\Ext^i_{\scrA}(A,B)=0$ for arbitrary objects $A,B\in \scrE$ and all $i>0$. Then the natural functor $\K(\scrE,\Phi,w)\to \Dco(\scrA,\Phi,w)$ is fully faithful.
\begin{proof}
Let $E,F\in \Fact(\scrE,\Phi,w)$ be objects. We need to show that the map
\begin{gather*}
Q\colon\Hom_{{\rm K}(\scrA,\Phi,w)}(E,F)\to \Hom_{\Dco(\scrA,\Phi,w)}(E,F)
\end{gather*}
 defined by the Verdier quotient $Q\colon{\rm K}(\scrA,\Phi,w)\to\Dco(\scrA,\Phi,w)$ is an isomorphism. By~\cite[Proposition~2.19]{bdfik} we can take an injective resolution $\iota\colon F\to (I^{\bullet},d_{I}^{\bullet})=\big(\cdots\to0\to I^0\xrightarrow{d_{I}^0} I^1\xrightarrow{d_{I}^1} \cdots\big)$ of $F$ in $\Ch(\Fact(\scrA,\Phi,w))$ in the sense that all $I^i$ has injective components. Then, we have the induced map
 \begin{gather*}
 \Tot(\iota)\colon\ F\to\Tot(I^{\bullet})
 \end{gather*}
 in $\Fact(\scrA,\Phi,w)$ which becomes an isomorphism in $\Dco(\scrA,\Phi,w)$. Consider the following commutative diagram:
\[
\begin{tikzcd}
\Hom_{{\rm K}(\scrA,\Phi,w)}(E,F)\arrow[rr, "Q"]\arrow[d, "\Tot(\iota)_*"']&&\Hom_{\Dco(\scrA,\Phi,w)}(E,F)\arrow[d, "\Tot(\iota)_*"]\\
\Hom_{{\rm K}(\scrA,\Phi,w)}(E,\Tot(I^{\bullet}))\arrow[rr, "Q"]&&\Hom_{\Dco(\scrA,\Phi,w)}(E,\Tot(I^{\bullet})).
\end{tikzcd}
\]
In the above diagram, the vertical arrow on the right hand side is an isomorphism, and the horizontal arrow on the bottom is also an isomorphism by~\cite[Lemma~2.24]{bdfik} and~\cite[Proposition~B.2(I)]{ls} (see also~\cite[Remark 2.14]{ls}). Hence it suffices to show that the vertical arrow on the left hand side is an isomorphism, and for this we prove that the map
\begin{equation}\label{hom cpx}
\Tot(\iota)_*\colon\ \Hom_{\dgFact(\scrA,\Phi,w)}(E,F) \to\Hom_{\dgFact(\scrA,\Phi,w)}(E,\Tot(I^{\bullet}))
\end{equation}
of cochain complexes is a quasi-isomorphism. The complex $\Hom_{\dgFact(\scrA,\Phi,w)}(E,F)$ can be seen as the double complex $X^{\bullet,\bullet}$ defined by
\begin{equation*}
X^{p,q}\defeq\begin{cases}
\Hom_{\dgFact(\scrA,\Phi,w)}(E,F)^p, & q=0, \\
0,& q\neq 0.
\end{cases}
\end{equation*}
On the other hand, we can consider the double complex $Y^{\bullet,\bullet}$ defined by
\begin{equation*}
Y^{p,q}\defeq\Hom_{\dgFact(\scrA,\Phi,w)}(E,I^q)^p,
\end{equation*}
 where the differentials $d_Y^{p,\bullet}\colon Y^{p,\bullet}\to Y^{p,\bullet+1}$ and $d_Y^{\bullet,q}\colon Y^{\bullet,q}\to Y^{\bullet+1,q}$ are given by $d_Y^{p,\bullet}\defeq (d_I^{\bullet})_*$ and $d_Y^{\bullet,q}\defeq d_{(E,I^q)}^{\bullet}$, where $d_{(E,I^q)}^{\bullet}$ is the differential of the complex $\Hom_{\dgFact(\scrA,\Phi,w)}(E,I^{q})^{\bullet}$. Then the injective map $\iota\colon F\to I^0$ defining the resolution $\iota\colon F\to I^{\bullet}$ gives the morphism
 \begin{equation}\label{hom cpx 2}
 \iota_*\colon\ X^{\bullet,\bullet}\to Y^{\bullet,\bullet}
 \end{equation}
 of double complexes. By definition, we have isomorphisms $\Tot(X^{\bullet,\bullet})=\Hom_{\dgFact(\scrA,\Phi,w)}(E,F)$ and $\Tot(Y^{\bullet,\bullet})\cong\Hom_{\dgFact(\scrA,\Phi,w)}(E,\Tot(I^{\bullet}))$, and the map $\Tot(\iota)_*$ in (\ref{hom cpx}) is the totalization $\Tot(\iota_*)$ of the map $\iota_*$ in (\ref{hom cpx 2}).
By Lemma~\ref{qis} the map $\iota_*\colon X^{p,\bullet}\to Y^{p,\bullet}$ of cochain complexes is a quasi-isomorphism for any $p\in \bZ$, and therefore the map in (\ref{hom cpx}) is also a quasi-isomorphism by~\cite[Theorem~1.9.3]{ks} (see also~\cite[Lemma~2.46]{ls}). This completes the proof.
\end{proof}
\end{lem}

\begin{lem}\label{canonical truncation}
Let $F^{\bullet}\in \D^+(\scrA)$ be an object, and let $m\in \bZ$ be an integer such that $F^k=0$ for any $k<m$. Then there are a family $\{F^{\bullet}_i\}_{i\in \bZ}$ of objects in $\Db(\scrA)$ with $F^{k}_i=0$ for any $k<m$ and $i$ and an exact triangle
\begin{gather*}
\bigoplus_{i\in \bZ}F^{\bullet}_i\to \bigoplus_{i\in \bZ}F^{\bullet}_i\to F^{\bullet}\to \bigoplus_{i\in \bZ}F^{\bullet}_i[1].
\end{gather*}
\begin{proof}
Denote by $\tau_{\leq i}(F^{\bullet})$ the canonical truncation of $F^{\bullet}$ defined by $\tau_{\leq i}(F^{\bullet})^j=F^j$ if $j<i$, $\tau_{\leq i}(F^{\bullet})^i=\Ker(d^i\colon F^i\to F^{i+1})$ and $\tau_{\leq i}(F^{\bullet})^j=0$ if $j>i$. If we set $F_i^{\bullet}\defeq \tau_{\leq i}(F^{\bullet})$, $F_i^{\bullet}$ lies in~$\Db(\scrA)$ and $F^{k}_i=0$ for any $k<m$ and $i$. The object $F^{\bullet}$ is the cokernel of an injective morphism
\begin{gather*}
\bigoplus_{i\in \bZ} F^{\bullet}_i\to \bigoplus_{i\in \bZ} F^{\bullet}_i
\end{gather*}
in the abelian category $\Ch(\scrA)$ given by
 $\id \oplus (-\iota_i)\colon F^{\bullet}_i\to F^{\bullet}_i\oplus F^{\bullet}_{i+1}$, where $\iota_i\colon F^{\bullet}_i\to F^{\bullet}_{i+1}$ is a~natural injection. Hence we have a triangle
\begin{gather*}
\bigoplus_{i\in \bZ}F^{\bullet}_i\to \bigoplus_{i\in \bZ}F^{\bullet}_i\to F^{\bullet}\to \bigoplus_{i\in \bZ}F^{\bullet}_i[1]
\end{gather*}
in $\D^+(\scrA)$.
\end{proof}
\end{lem}

Let $\scrB$ be a Grothendieck category with enough projectives, and let $F\colon \scrA\to\scrB$ be an addi\-tive functor such that it commutes with small direct sums and has a left adjoint functor $G\colon \scrB\to \scrA$. The adjunction $G\dashv F$ implies that $F$ is left exact and $G$ is right exact. By~\cite[Proposition~14.3.4]{ks2}, $F$ admits a right derived functor
 \begin{gather*}
 \bR F\colon\ \D(\scrA)\to \D(\scrB),
 \end{gather*}
 which commutes with small direct sums, and $ \bR F$ restricts to a functor $\bR^+\!F\colon \D^+(\scrA)\to \D^+(\scrB)$. On the other hand, by~\cite[Theorem~14.4.3]{ks2}, $G$ admits a left derived functor
 \begin{gather*}
 \bL G\colon\ \D(\scrB)\to \D(\scrA),
 \end{gather*}
 and by~\cite[Theorem~14.4.5]{ks2} the adjunction $G\dashv F$ induces an adjunction
 \begin{gather*}
 \bL G\dashv \bR F.
 \end{gather*}

\begin{lem}\label{above ff}
 Assume that $\bR^+\!F\colon \D^+(\scrA)\to \D^+(\scrB)$ restricts to the functor $\bR F\colon \Db(\scrA)\to\Db(\scrB)$ and that there exists a positive integer $d>0$ such that for every object $B\in \scrB$ the projective dimension of $B$ is less than $d$.
 Then, we have the following:
 \begin{itemize}\itemsep=0pt
 \item[$1.$] The functor $\bL G$ restricts to a functor $\bL^{+}G\colon \D^+(\scrB)\to \D^+(\scrA)$.
 \item[$2.$] If $\bR F\colon \Db(\scrA)\to\Db(\scrB)$ is fully faithful, so is $\bR^+\!F\colon \D^+(\scrA)\to \D^+(\scrB)$.
 \end{itemize}
\begin{proof}
(1)
Let $B^{\bullet}\in \D^+(\scrB)$ be an object, and let $m\in\bZ$ be an integer such that $B^k=0$ for any $k<m$. Then by Lemma~\ref{canonical truncation} there are a family $\{B^{\bullet}_i\}_{i\in\bZ}$ of objects in $\Db(\scrB)$ with $B_i^k=0$ for any $k<m$ and $i$ and an exact sequence
\begin{equation}\label{truncation triangle}
\bigoplus_{i\in \bZ}B^{\bullet}_i\to \bigoplus_{i\in \bZ}B^{\bullet}_i\to B^{\bullet}\to \bigoplus_{i\in \bZ}B^{\bullet}_i[1].
\end{equation}
We set $C_i^{\bullet}\defeq\bL G(B^{\bullet}_i)$. Then, by the assumption, $C_i^{k}=0$ for any $k<m-d$ and $i$, and so $\bigoplus_{i\in \bZ}C_i^{\bullet}$ lies in $\D^+(\scrA)$.
Since $\bL G$ admits a right adjoint functor, it commutes with small direct sums.
Applying the functor $\bL G$ to the triangle \eqref{truncation triangle}, we obtain a triangle
\begin{gather*}
\bigoplus_{i\in \bZ}C^{\bullet}_i\to \bigoplus_{i\in \bZ}C^{\bullet}_i\to \bL G(B^{\bullet})\to \bigoplus_{i\in \bZ}C^{\bullet}_i[1],
\end{gather*}
which implies that $\bL G(B^{\bullet})\in\D^+(\scrA)$.

(2) Since $\bR^+\!F$ admits a left adjoint functor $\bL^+G$ by (1), it is enough to show that the adjunction morphism
\begin{gather*}
\varepsilon\colon\ \Upphi\defeq \bL^+G\circ \bR^+\!F\to \id_{\D^+(\scrA)}
\end{gather*}
is an isomorphism of functors. Let $A^{\bullet}\in \D^+(\scrA)$ be an object. Then by Lemma~\ref{canonical truncation} there are a family $\{A^{\bullet}_i\}_{i\in\bZ}$ of objects in $\Db(\scrA)$ and an exact triangle
\begin{gather*}
\bigoplus_{i\in \bZ}A^{\bullet}_i\xrightarrow{f} \bigoplus_{i\in \bZ}A^{\bullet}_i\xrightarrow{g} A^{\bullet}\to \bigoplus_{i\in \bZ}A^{\bullet}_i[1].
\end{gather*}
Consider the following commutative diagram
\[
\begin{tikzcd}
\Upphi\big(\bigoplus_{i\in \bZ}A^{\bullet}_i\big)\arrow[rr, "\Upphi(f)"]\arrow[d, "\varepsilon"']&&\Upphi\big(\bigoplus_{i\in \bZ}A^{\bullet}_i\big)\arrow[rr, "\Upphi(g)"]\arrow[d, "\varepsilon"']&&\Upphi(A^{\bullet})\arrow[d, "\varepsilon"]
\\
\bigoplus_{i\in \bZ}A^{\bullet}_i\arrow[rr, "f"]&&\bigoplus_{i\in \bZ}A^{\bullet}_i\arrow[rr, "g"]&&A^{\bullet},
\end{tikzcd}
\]
where the horizontal sequences are triangles and the vertical arrows are induced by $\varepsilon$. Since $\Upphi$ commutes with direct sums, the vertical arrows on the left and middle are the direct sums of morphisms $\varepsilon_i\defeq\varepsilon(A_i^{\bullet})\colon\Upphi(A_i^{\bullet})\to A_i^{\bullet}$, and each $\varepsilon_i$ is an isomorphism since $\bR F\colon \Db(\scrA)\to\Db(\scrB)$ is fully faithful and $A_i^{\bullet}\in\Db(\scrA)$. Hence the vertical arrow on the right hand side is also an~iso\-morphism.
\end{proof}
\end{lem}

Now we are ready to prove that a tilting object induces an equivalence of derived factorization categories. Let $T\in \scrA$ be a partial tilting object such that the exact subcategory $\Add T\subset \scrA$ is preserved by $\Phi$. Set $\Lambda\defeq\End_{\scrA}(T)$, and
 let $(\Psi,v)$ be a potential on $\Mod\Lambda$. We define
 \begin{gather*}
 \DcoMod(\Lambda,\Psi,v)\defeq\Dco(\Mod\Lambda,\Psi,v),
 \\
 \Dmod(\Lambda,\Psi,v)\defeq\D(\fmod\Lambda,\Psi,v).
 \end{gather*}
Consider the functor
 \begin{gather*}
 \F\defeq\Hom_{\scrA}(T,-)\colon\ \scrA\to \Mod\Lambda,
 \end{gather*}
 and assume that there is a left adjoint functor $\G\colon \Mod\Lambda\to \scrA$, and that $\F$ and $\G$ are factored with respect to $(\Phi,w)$ and $(\Psi,v)$. Let
 \begin{gather*}
 \scrC\subset \scrA
 \end{gather*}
 be an abelian subcategory preserved by $\Phi$ such that $\add T\!\subset \scrC$, $\F(\scrC)\!\subset \fmod\Lambda$ and \mbox{$\G(\fmod\Lambda)\subset \scrC$}. Assume that the natural functors $\Db(\scrC)\to\Db(\scrA)$ and $\D(\scrC,\Phi,w)\to \Dco(\scrA,\Phi,w)$ are fully faithful and that $\bR \F\colon\Db(\scrA)\to\Db(\Mod\Lambda)$ restricts to $\bR \F\colon\Db(\scrC)\to\Db(\fmod\Lambda)$. Then $\bR \F\colon\Dco(\scrA,\Phi,w)\to \DcoMod(\Lambda,\Psi,v)$ restricts to the functor
\begin{gather*}
\bR \F\colon\ \D(\scrC,\Phi,w)\to \Dmod(\Lambda,\Psi,v).
\end{gather*}
Moreover we assume that $\Lambda$ is of finite global dimension. Then we have the left derived functor
\begin{gather*}
\bL\G\colon\ \Dmod(\Lambda,\Psi,v)\to \D(\scrC,\Phi,w)
\end{gather*}
that is left adjoint to $\bR\F\colon\D(\scrC,\Phi,w)\to \Dmod(\Lambda,\Psi,v)$.
\begin{thm}\label{main thm}

Suppose that the adjunction morphisms $\Lambda\to \F\G(\Lambda)$ and $\G\F(T)\to T$ are isomorphisms.
\begin{itemize}\itemsep=0pt
\item[$1.$] The functor $\bL \G\colon \Dmod(\Lambda,\Psi,v)\to \D(\scrC,\Phi,w)$ is fully faithful, and it factors through an equivalence
\begin{gather*}
\bL \G\colon \Dmod(\Lambda,\Psi,v)\simto \Kadd(T,\Phi,w).
\end{gather*}

\item[$2.$]Assume that every object in $\scrA$ has finite injective resolution and that the right derived functor $\bR \F\colon\Db(\scrA)\to\Db(\Mod\Lambda)$ is an equivalence. Then the functor
\begin{gather*}
\bR \F\colon\ \Dco(\scrA,\Phi,w)\to \DcoMod(\Lambda,\Psi,v)
\end{gather*}
is an equivalence, and it restricts to the equivalence
\begin{gather*}
\bR \F\colon\ \D(\scrC,\Phi,w)\simto \Dmod(\Lambda,\Psi,v).
\end{gather*}
\end{itemize}
\begin{proof}
(1) Recall that the functor $\bL\G\colon\Dmod(\Lambda,\Psi,v)\to \D(\scrC,\Phi,w)$ is the composition
\[
\begin{tikzcd}
\Dmod(\Lambda,\Psi,v)\arrow[r,"\sim"]&\K(\proj\Lambda,\Psi,v)\arrow[r,"\G"] &\Kadd(T,\Phi,w)\arrow[r,"Q"]&\D(\scrC,\Phi,w),
\end{tikzcd}
\]
where the first equivalence follows from Proposition~\ref{proj resol} and $Q$ is the natural quotient functor. Since we have the equivalence $\G\colon \proj\Lambda\simto \add T$, the middle functor $\G$ is also an equivalence. Thus the statement follows from Lemma~\ref{add ff} and the assumption that the natural functor $\D(\scrC,\Phi,w)\to \Dco(\scrA,\Phi,w)$ is fully faithful.

(2) Since $(\bR \F)^+\colon\D^+(\scrA)\to\D^+(\Mod\Lambda)$ is fully faithful by Lemma~\ref{above ff}, the functor $\bR \F $: $\Dco(\scrA,\Phi,w)\to \DcoMod(\Lambda,\Psi,v)$ is also fully faithful by~\cite[Lemma~4.11]{bdfik}. Consider the following commutative diagram:
\[
\begin{tikzcd}
\K\Add(T,\Phi,w)\arrow[rr, "\F"]\arrow[d, ]&&\K\Proj(\Lambda,\Psi,v)\arrow[d, ]\\
\Dco(\scrA,\Phi,w)\arrow[rr, "\bR \F"]&&\DcoMod(\Lambda,\Psi,v).
\end{tikzcd}
\]
Since the additive functor $\F\colon\Add(T)\to \Proj\Lambda$ is an equivalence, the top horizontal arrow is an equivalence. Moreover, the vertical arrow on the right hand side is essentially surjective by~Proposition~\ref{proj resol}. Hence $\bR \F\colon\Dco(\scrA,\Phi,w)\to \DcoMod(\Lambda,\Psi,v)$ is essentially surjective, and so it is an equivalence.

Since the natural embedding functors $\D(\scrC,\Phi,w)\to \Dco(\scrA,\Phi,w)$ and $\Dmod(\Lambda,\Psi,v)\to \DcoMod(\Lambda,\Psi,v)$ are fully faithful, the functor $\bR \F\colon\!\D(\scrC,\Phi,w)\!\to \Dmod(\Lambda,\Psi,v)$ and its left~adjo\-int $\bL \G\colon \! \Dmod(\Lambda,\Psi,v)\to \D(\scrE,\Phi,w)$ are fully faithful. Hence $\bR \F\colon\!\D(\scrC,\Phi,w)\to
\Dmod(\Lambda,\Psi,v)$ is also an equivalence.
\end{proof}
\end{thm}

\subsection{Derived factorization categories of noncommutative gauged LG models}

In this short subsection, we define (noncommutative) gauged Landau--Ginzburg models, and its derived factorization categories.

\begin{dfn}
We call the data $(\cA,L,w)^G$ a {\it gauged Landau--Ginzburg model} when $G$ is an~alge\-braic group acting on a scheme $X$, $\cA$ is a $G$-equivariant coherent $X$-algebra, $L$ is a $G$-equivariant line bundle on $X$, and $w\in \Gamma(X,L)^G$ is a $G$-invariant global section of $L$. A gauged Landau--Ginzburg model $(\cA,L,w)^G$ is said to be {\it noncommutative}, if the algebra $\cA$ is not commutative. We write $(X,L,w)^G\defeq(\cO_X,L,w)^G$, and for a character $\chi\colon G\to \bG_m$ of $G$ we write $(\cA,\chi,w)^G\defeq(\cA,\cO(\chi),w)^G$.
\end{dfn}

If $(\cA,L,w)^G$ is a gauged LG model, we have the triple
\begin{gather*}
\big(\Qcoh_G\cA,L\otimes_{\cO_X}(-),w\big)
\end{gather*}
as in \eqref{triple}, where $L\otimes_{\cO_X}(-)\colon \Qcoh_G\cA\simto \Qcoh_G\cA$ is the tensor product with $L$, and $w\colon\id\to L\otimes_{\cO_X}(-)$ is the functor morphism defined by the multiplication by $w$. Then we define
\begin{gather*}
\Qcoh_G(\cA,L,w)\defeq\Fact\big(\Qcoh_G\cA,L\otimes_{\cO_X}(-),w\big),
\\[.5ex]
\DcoQcoh_G(\cA,L,w)\defeq\Dco\big(\Qcoh_G\cA,L\otimes_{\cO_X}(-),w\big),
\\[.5ex]
\Dcoh_G(\cA,L,w)\defeq\D\big(\coh_G\cA,L\otimes_{\cO_X}(-),w\big),
\\[.5ex]
\DMF_G(\cA,L,w)\defeq\D\big(\lfr_G\cA,L\otimes_{\cO_X}(-),w\big),
\end{gather*}
{\sloppy
where $\lfr_G\cA$ is the subcategory of $\coh_G\cA$ consisting of locally free $\cA$-modules. We call $\Dcoh_G(\cA,L,w)$ the {\it derived factorization category} of $(\cA,L,w)^G$, and $\DMF_G(\cA,L,w)$ the {\it derived matrix factorization category} of $(\cA,L,w)^G$. If $X=\Spec R$ is an affine scheme and $A\defeq\Gamma(X,\cA)$ is the corresponding $G$-equivariant $R$-algebra, we define
\begin{gather*}
\Mod_G(A,L,w)\defeq\Fact\big(\Mod_GA,L\otimes_{R}(-),w\big),
\\[.5ex]
\DcoMod_G(A,L,w)\defeq\Dco\big(\Mod_G\cA,L\otimes_{R}(-),w\big),
\\[.5ex]
\Dmod_G(A,L,w)\defeq\D\big(\fmod_GA,L\otimes_{R}(-),w\big),
\\[.5ex]
\KMF_G(A,L,w)\defeq\K\big(\proj_GA,L\otimes_{R}(-),w\big).
\end{gather*}

}

\subsection{Equivariant tilting modules and factorizations}
In this subsection, we apply the general result in Section~\ref{general morita} to the following geometric setting.

Let the notation be the same as in Section~\ref{derived section}.
Let $\chi\colon G/H\to \bG_m$ be a character of $G/H$, and set $\widehat{\chi}\defeq\chi\circ p\colon G\to \bG_m$, where $p\colon G\to G/H$ is the natural projection. Write for $\cO_R(\chi)$ (resp.\ $\cO_X\big(\widehat{\chi}\big)$) the associated $G/H$-equivariant invertible sheaf on $\Spec R$ (resp.\ $G$-equivariant invertible sheaf on $X$). Take a $G/H$-invariant global section $w_R\in \Gamma(\Spec R,\cO_R(\chi))^{G/H}$ of~$\cO_R(\chi)$, and set $w\defeq (f)^*_{p}w_R\in \Gamma\big(X,\cO_X\big(\widehat{\chi}\big)\big)^G$, where $(f)^*_{p}\colon \Mod_{G/H}R\to \Qcoh_GX$ is the pull-back by the $p$-equivariant morphism $f$.
Then we have the gauged LG models
\begin{gather*}
\big(\cA,\widehat{\chi},w\big)^G \qquad \text{and}\qquad \big(\Lambda^H,\chi,w_R\big)^{G/H},
\end{gather*}
and we define
\begin{gather*}
\Kadd_G\big(\cT,\widehat{\chi},w\big)\defeq \K\big(\add_{\widehat{\chi}}\cT,\cO\big(\widehat{\chi}\big)\otimes_{\cO_X}(-),w\big),
\end{gather*}
where $\add_{\widehat{\chi}}\cT$ is the subcategory of $\coh_G\cA$ defined by the following additive closure
\begin{gather*}
\add_{\widehat{\chi}}\cT\defeq\add\big\{\cO\big(\widehat{\chi}^n\big)\otimes_{\cO_X} \big(\cT,\uptheta^{\cT}\big)\mid n\in \bZ \big\}.
\end{gather*}

Since the functors $\F\colon \Qcoh_{G}\cA\to \Mod_{G/H}\Lambda^H$ and $\G\colon \Mod_{G/H}\Lambda^H\to \Qcoh_{G}\cA$ are factored with respect to $(\widehat{\chi},w)$ and $(\chi,w_R)$, they induce the functors
\begin{gather*}
\F\colon\ \Qcoh_G\big(\cA,\widehat{\chi},w\big)\to\Mod_{G/H}\big(\Lambda^H,\chi,w_R\big),
\\
\G\colon\ \Mod_{G/H}\big(\Lambda^H,\chi,w_R\big)\to \Qcoh_G\big(\cA,\widehat{\chi},w\big).
\end{gather*}
Since $\Qcoh_G\cA$ has enough injectives, we have the right derived functor
\begin{gather*}
\bR\F\colon\ \DcoQcoh_G\big(\cA,\widehat{\chi},w\big)\to\DcoMod_{G/H}\big(\Lambda^H,\chi,w_R\big),
\end{gather*}
and it restricts to the functor
\begin{gather*}
\bR\F\colon\ \Dcoh_G\big(\cA,\widehat{\chi},w\big)\to\Dmod_{G/H}\big(\Lambda^H,\chi,w_R\big).
\end{gather*}
If $G/H$ is reductive and every $\Lambda^H$-module has a finite projective resolution, then every finitely generated $G/H$-equivariant $\Lambda^H$-module has a finite projective resolution in $\fmod_{G/H}\Lambda^H$ by Lemma~\ref{finite proj resol}. In this case, we have the left derived functor
\begin{gather*}
\bL\G\colon\ \Dmod_{G/H}\big(\Lambda^H,\chi,w_R\big)\to \Dcoh_{G}\big(\cA,\widehat{\chi},w\big).
\end{gather*}

\begin{thm}\label{tilt and fact}
Assume that $G/H$ is reductive and $\Lambda^H$ is of finite global dimension.
\begin{itemize}\itemsep=0pt
\item[$1.$] If $\big(\cT,\uptheta^{\cT}\big)$ is partial $(G,H)$-tilting, then the functor
\begin{gather*}
\bL\G\colon\ \Dmod_{G/H}\big(\Lambda^H,\chi,w_R\big)\to \Dcoh_{G}\big(\cA,\widehat{\chi},w\big)
\end{gather*}
is fully faithful, and it restricts to the equivalence
\begin{gather*}
\bL\G\colon\ \Dmod_{G/H}\big(\Lambda^H,\chi,w_R\big)\simto \Kadd_{G}\big(\cT,\widehat{\chi},w\big).
\end{gather*}
\item[$2.$] If $\big(\cT,\uptheta^{\cT}\big)$ is $(G,H)$-tilting and every $\cM\in \Qcoh\cA$ has a finite injective resolution in $\Qcoh\cA$, then we have the equivalence
\begin{gather*}
\bR\F\colon\ \Dcoh_{G}\big(\cA,\widehat{\chi},w\big)\simto \Dmod_{G/H}\big(\Lambda^H,\chi,w_R\big).
\end{gather*}
\end{itemize}
\begin{proof}
By an identical argument as in the proof of Lemma~\ref{faithful lemma}, we see that the restriction functors
\begin{gather*}
\Res^G_H\colon\ \Dcoh_G\big(\cA,\widehat{\chi},w\big)\to \Dcoh_H(\cA,\chi,w),
\\
\Res\colon\ \Dmod_{G/H}\big(\Lambda^H,\chi,w_R\big)\to \Dmod\big(\Lambda^H,\id,w_R\big)
\end{gather*}
are faithful. Let $\F_H\colon \coh_H\cA\to\fmod\Lambda^H$ and $\G_H\colon\fmod\Lambda^H\to\coh_H\cA$
be the same functors as in the proof of Theorem~\ref{derived equiv}, and recall that the adjunction morphisms
$\Lambda^H\to \F_H\G_H\big(\Lambda^H\big)$ and $\G_H\F_H(\cT_H)\to \cT_H$
are isomorphisms. Then these functors define the functors
\begin{gather*}
\bR\F_H\colon\ \Dcoh_H\big(\cA,\widehat{\chi},w\big)\to \Dmod\big(\Lambda^H,\chi,w_R\big),
\\
\bL\G_H\colon\ \Dmod\big(\Lambda^H,\chi,w_R\big)\to \Dcoh_H\big(\cA,\widehat{\chi},w\big),
\end{gather*}
and if $\big(\cT,\uptheta^{\cT}\big)$ is partial $(G,H)$-tilting (resp.\ $(G,H)$-tilting), $\bL\G_H$ is fully faithful (resp.\ an~equi\-valence) by Theorem~\ref{main thm}.
Hence the results follow from Lemma~\ref{reduction lemma}.
\end{proof}
\end{thm}


\section[Linear sections of Pfaffian varieties and noncommutative resolutions]
{Linear sections of Pfaffian varieties\\ and noncommutative resolutions}

In this section, we prove that derived categories of noncommutative resolutions of linear sections of Pfaffian varieties are equivalent to the derived factorization categories of noncommutative gauged LG models.

\subsection{Noncommutative resolutions of Pfaffian varieties and its linear sections} Following~\cite{rs}, we recall noncommutative resolutions of Pfaffian varieties and its linear sections.
Let $V$ be a vector space of dimension $v$. For an integer $q$ with $0\leq 2q\leq v$, we have a Pfaffian variety
\begin{gather*}
\Pf_q\defeq\big\{x\in\gwedge V^* \mid \rk(x)\leq 2q\big\}\subseteq \bP\big(\gwedge V^*\big),
\end{gather*}
and we denote its affine cone by
\begin{gather*}
\Pf_q^{\aff}\subseteq \gwedge V^*.
\end{gather*}
Then we have $\dim \Pf_q^{\aff}=q(2v-2q-1)$.
The Pfaffian variety $\Pf_q$ is smooth if and only if~$q=1$, where it defines a Grassmannian $\Gr(2,V)$, or $q=\lfloor v/2\rfloor$, where it defines the whole space $\bP\big(\gwedge V^*\big)$. In other cases, the singular locus is the subvairiety $\Pf_{(q-1)}\subset \Pf_q$. Let $L\subset \gwedge V^*$ be a subspace of codimension $c$ such that
\begin{gather*}
\Pf_q\rest\defeq\Pf_q\cap\,\,\bP(L)\neq\varnothing
\end{gather*}
and $\Pf_q\rest$ has the expected dimension $\dim\Pf_q-c$. If $c>\dim \Sing(\Pf_q)=(q-1)(2v-2q+1)-1$, we can take a generic $L$ so that $\Pf_q\rest$ is smooth. If $c$ is smaller than the bounds, then $\Pf_q\rest$ is never smooth, and in this case the usual bounded derived category $\Db\big(\coh \Pf_q\rest\big)$ does not behave well (for example, in the context of homological projective duality~\cite{rs}). 

 Let $(Q,\omega)$ be a symplectic vector space of dimension $2q$ with a symplectic form $\omega\in \gwedge Q^{\ast}$. Let
\begin{gather*}
\Sp(Q)\defeq\{f\in \GL(Q)\mid f^*\omega=\omega\}
\end{gather*}
be the symplectic group of $Q$, and
\begin{gather*}
\GSp(Q)\defeq\{f\in \GL(Q)\mid \exists\, t\in k^*\mbox{ such that } f^*\omega=t\, \omega\}
\end{gather*}
the symplectic similitude group of $Q$.
Then we have a short exact sequence
\begin{equation}\label{ses}
1\to\Sp(Q)\hookto \GSp(Q)\surjto \bG_m\to1
\end{equation}
that induces a semi-direct product $\GSp(Q)= \Sp(Q)\rtimes\bG_m$, where $\bG_m\subset \GSp(Q)$ is the diagonal subgroup. Let us set
\begin{gather*}
Y\defeq\Hom_k(V,Q)
\end{gather*}
and consider the following quotient stacks
\begin{gather*}
\cY^{\aff}\defeq[Y/\Sp(Q)]
\qquad\text{and}\qquad
\cY\defeq[Y/\GSp(Q)].
\end{gather*}
Then the surjective morphism
\begin{gather*}
p\colon\quad Y\to \Pf_q^{\aff}, \qquad \varphi\mapsto \varphi^*\omega
\end{gather*}
induces the morphisms
\begin{gather*}
\pi^{\aff}\colon\ \cY^{\aff}\to \Pf_q^{\aff}
\qquad\text{and}\qquad
\pi\colon\ \cY\to \big[\Pf_q^{\aff}/\bG_m\big],
\end{gather*}
and if we set $\cY^{\stab}\defeq\left[\left(Y\backslash \,p^{-1}(0)\right)/\GSp(Q)\right]$, the morphism $\pi\colon \cY\to \big[\Pf_q^{\aff}/\bG_m\big]$ restricts to a~(stacky) resolution
\begin{gather*}
\pi\colon\ \cY^{\stab}\to \Pf_q
\end{gather*}
of the Pfaffian variety $\Pf_q$.

Recall that irreducible representations of $\Sp(Q)$ are indexed by Young diagrams of height at most $q$. We denote by $Y_{q,s}$ the set of Young diagrams of height at most $q$ and width at most $s\defeq\lfloor v/2\rfloor-q$. Since vector bundles on $\cY^{\aff}$ are nothing but $\Sp(Q)$-equivariant vector bundles on $Y$, each representation of $\Sp(Q)$ defines a vector bundle
 on $\cY^{\aff}$. For each Young diagram $\gamma\in Y_{q,s}$, we can choose a vector bundle $\scrV_{\gamma}\in \vect\cY$ whose pull-back to $\cY^{\aff}$ by the natural projection $\cY^{\aff}\to \cY$ is the vector bundle associated to $\gamma$ (see~\cite[Section~2.4]{rs}). Then the vector bundle
\begin{gather*}
\scrV:=\bigoplus_{\gamma\in Y_{q,s}}\scrV_{\gamma}\in \coh\cY
\end{gather*}
is a partial tilting bundle on $\cY$. Restricting $\scrV$ to the open substack $\cY^{\stab}$ we obtain a partial tilting bundle
$\scrV^{\stab}$ on $\cY^{\stab}$, and the ${\Pf_q}$-algebra
\begin{gather*}
\cB\defeq\pi_*\cEnd_{\cY^{\stab}}(\scrV^{\stab})
\end{gather*}
is a noncommutative resolution of $\Pf_q$ by~\cite[Section~5]{svdb}.
By the noncommutative Bertini theorem~\cite{rsv}, for a generic $L$ the restriction
\begin{gather*}
\cB_L:=\cB|_{\Pf_q\rest}
\end{gather*}
 is a noncommutative resolution of the linear section $\Pf_q\rest$. If $\Pf_q\rest$ is smooth, the category $\Db (\coh \cB_L )$ is equivalent to $\Db (\coh \Pf_q\rest )$.

 \subsection[Noncommutative resolutions of Pfaffian varieties and noncommutative gauged LG models]
 {Noncommutative resolutions of Pfaffian varieties \\and noncommutative gauged LG models}

 In this subsection, we prove that the derived category of a noncommutative resolution of a linear section of a Pfaffian variety is equivalent to the derived factorization category of a noncommutative gauged Landau--Ginzburg model.

 We use the same notation as in the previous section. We recall that the category $\Db(\coh \cB_L)$ is equivalent to a full subcategory of the derived factorization category of a gauged LG model. We write $L^{\perp}\subset\gwedge V$ for the annihilator of $L$. We have a $\GSp(Q)$-action on the product $Y\times L^{\perp}$ defined by
 \begin{gather*}
 \GSp(Q)\times \big(Y\times L^{\perp}\big)\to \big(Y\times L^{\perp}\big), \qquad (f, \varphi,x)\mapsto \big(f\circ \varphi, \sigma(f)^{-1}x\big),
 \end{gather*}
 where $\sigma\colon \GSp(Q)\to\bG_m$ is the surjection in \eqref{ses}. Consider the following quotient stacks
\begin{gather*}
 \cY\times_{\bG_m}L^{\perp}\defeq\big[Y\times L^{\perp}/\GSp(Q)\big],
 \\
 \cY^{\stab}\times_{\bG_m}L^{\perp}\defeq\big[\big(Y\backslash\,p^{-1}(0)\big)\times L^{\perp}/\GSp(Q)\big].
 \end{gather*}
 The function $W\colon Y\times L^{\perp}\to \bC$ defined by $W(\varphi,x)\defeq(\varphi^*\omega)(x)$ is $\GSp(Q)$-invariant, and thus it defines a potential
 \begin{gather*}
 w\colon\ \cY\times_{\bG_m}L^{\perp}\to \bC.
 \end{gather*}
 We define an additional $\bG_m$-action, which is called an {\it R-charge}, on $Y\times L^{\perp}$ by
 \begin{gather*}
 \bG_m\times \big(Y\times L^{\perp}\big)\to \big(Y\times L^{\perp}\big),\qquad
 (t, \varphi,x)\mapsto (t\varphi, x).
 \end{gather*}
This $\bG_m$-action commutes with the above $\GSp(Q)$-action, and so we have an induced $\bG_m$-action on $\cY\times_{\bG_m}L^{\perp}$. Then the potential $w\colon \cY\times_{\bG_m}L^{\perp}\to \bC$ is a semi-invariant regular function with respect to the character $\chi_1\defeq\id_{\bG_m}\colon \bG_m\to \bG_m$. Following~\cite{rs}, we define the {\it B-brane category} $\DB\big(\cY\times_{\bG_m}L^{\perp},w\big)$
as follows:

 {\sloppy
 We write $\rep\GSp(Q)$ for the category of finite dimensional algebraic representations of~$\GSp(Q)$, and we define the full subcategory
\begin{gather*}
\scrY_{q,s}\subset \rep\GSp(Q)
\end{gather*}}\noindent
consisting of irreducible representations whose restriction to the subgroup $\Sp(Q)$ is an irreducible representation of $\Sp(Q)$ corresponding to some Young diagram in $Y_{q,s}$.
Since the origin $0\in Y\times L^{\perp}$ is a fixed point of the above action from $\GSp(Q)\times \bG_m$, the restriction to the origin defines a functor
\begin{gather*}
(-)|_0\colon\ \Dcoh_{\bG_m}\big(\cY\times_{\bG_m}L^{\perp},\chi_1,w\big)\to \Dcoh_{\bG_m}\big([0/\GSp(Q)],\chi_1,0\big)\simto\Db(\rep\GSp(Q)),
\end{gather*}
where the latter functor is a similarly equivalence as in~\cite[Proposition~2.14]{H2}.
Then we define the category
\begin{gather*}
\DB\big(\cY\times_{\bG_m}L^{\perp},w\big)
\end{gather*}
 to be the full subcategory of $\Dcoh_{\bG_m}\big(\cY\times_{\bG_m}L^{\perp},\chi_1,w\big)$ consisting of objects $F$ such that each cohomology $H^i(F|_0)$ of the restriction $F|_0\in \Db(\rep\GSp(Q))$ lies in $\scrY_{q,s}$. Restricting to the open substack $\cY^{\stab}\times_{\bG_m}L^{\perp}$, we have the B-brane subcategory
 \begin{gather*}
 \DB\big(\cY^{\stab}\times_{\bG_m}L^{\perp},w\big)\subset \Dcoh_{\bG_m}\big(\cY^{\stab}\times_{\bG_m}L^{\perp},\chi_1,w\big),
 \end{gather*}
 and by~\cite[Section~4.1]{rs} we have an equivalence
\begin{equation}\label{nc br}
\Db(\coh \cB_L)\cong {\rm DB}\big(\cY^{\rm ss}\times_{\bG_m}L^{\perp},w\big).
\end{equation}

For an interval $I\subset \bZ$, we define a subcategory
\begin{gather*}
\scrY_{q,s}^I\subseteq \scrY_{q,s}
\end{gather*}
of $\scrY_{q,s}$ consisting of representations whose restriction to the diagonal subgroup $\bG_m\subset \GSp(Q)$ has weights in $I$. This subcategory defines a full subcategory
\begin{gather*}
\DB\big(\cY\times_{\bG_m}L^{\perp},w\big)_I\subset\DB\big(\cY\times_{\bG_m}L^{\perp},w\big)
\end{gather*}
consisting of objects $F$ such that each cohomology $H^i(F|_0)$ of the restriction $F|_0$ lies in $\scrY_{q,s}^I$.
By~\cite[Theorem~4.7]{rs} we have an equivalence
\begin{equation}\label{window}
{\rm DB}\big(\cY^{\rm ss}\times_{\bG_m}L^{\perp},w\big)\cong {\rm DB}\big(\cY\times_{\bG_m}L^{\perp},w\big)_{[-qv,qv]}.
\end{equation}

We set
\begin{gather*}
\scrT:=\bigoplus_{\rho \in \scrY^{[-qv,qv]}_{q,s}} \scrT_{\rho}\in \coh_{\GSp(Q)}Y\times L^{\perp},
\end{gather*}
where $\scrT_{\rho}$ is the $\GSp(Q)$-equivariant vector bundle on $Y\times L^{\perp}$ associated to an irreducible representation $\rho\in\scrY^{[-qv,qv]}_{q,s}$, and choose a $\bG_m\times \GSp(Q)$-equivariant vector bundle
\begin{gather*}
\widetilde{\scrT}\in \coh_{\bG_m\times \GSp(Q)}Y\times L^{\perp}
\end{gather*}
on $Y\times L^{\perp}$ such that $\Res^{\bG_m\times \GSp(Q)}_{\GSp(Q)}\widetilde{\scrT}\cong \scrT$. Then the natural functor
\begin{gather*}
\Kadd_{\bG_m\times \GSp(Q)}\big(\widetilde{\scrT},\chi,w\big)\to \Dcoh_{\bG_m}\big(\cY\times_{\bG_m}L^{\perp},\chi_1,w\big)
\end{gather*}
is fully faithful by Lemma~\ref{add ff}, and by~\cite[Lemma~2.8]{rs} we have a natural equivalence
\begin{equation}\label{window add}
 {\rm DB}\big(\cY\times_{\bG_m}L^{\perp},w\big)_{[-qv,qv]}\cong \Kadd_{\bG_m\times \GSp(Q)}\big(\widetilde{\scrT},\chi,w\big).
\end{equation}
Since $\scrT$ is a partial tilting object in $\Qcoh _{\GSp(Q)}Y\times L^{\perp}$ by construction, $\widetilde{\scrT}$ is a partial $(\bG_m\times \GSp(Q),\GSp(Q))$-tilting bundle. If we set
\begin{gather*}
\Lambda\defeq\End_{\cY\times_{\bG_m}L^{\perp}}(\scrT)=\End_{Y\times L^{\perp}}(\scrT)^{\GSp(Q)},
\end{gather*}
then $\Lambda$ is a $\bG_m$-equivariant algebra over the affine quotient $\big(Y\times L^{\perp}\big)/\GSp(Q)$. Moreover, by the results in~\cite[Section~1.5]{svdb} and similar arguments as in~\cite[Section~5]{svdb}, we see that the algebra $\Lambda$ is a noncommutative resolution of $\big(Y\times L^{\perp}\big)/\GSp(Q)$.
By Theorem~\ref{tilt and fact}, we have an equivalence
\begin{gather}\label{tilt equiv}
\Kadd_{\bG_m\times \GSp(Q)}\big(\widetilde{\scrT},\chi,w\big)\simto\Dmod_{\bG_m}(\Lambda, \chi,w).
\end{gather}
Combining the equivalences \eqref{nc br}, \eqref{window}, \eqref{window add} and \eqref{tilt equiv} we have the following:

\begin{cor}\label{main cor}
Notation is the same as above. We have an exact equivalence
\begin{gather*}
\Db(\coh \cB_L)\cong \Dmod_{\bG_m}(\Lambda, \chi,w).
\end{gather*}
In particular, if $\Pf_q\rest$ is smooth, we have an equivalence
\begin{gather*}
\Db\big(\coh \Pf_q\rest\big)\cong \Dmod_{\bG_m}(\Lambda, \chi,w).
\end{gather*}
\end{cor}

\begin{rem}
Let $T\defeq\big\{\diag\big(z_1,\hdots,z_q,z_0z_1^{-1},\hdots,z_0z_q^{-1}\big)\mid z_i\in \bG_m\big\}\subset \GSp(Q)$ be a standard maximal torus of $\GSp(Q)$. Denote by $\chi_i\colon T\to \bG_m$ the character defined by
\begin{gather*}
(z_0,z_1,\hdots,z_q)\mapsto z_i.
\end{gather*}
Then the $T$-weights of $Y$ are $\{\chi_i,\chi_0-\chi_i\}_{1\leq i\leq q}$ with each weight occurring with multiplicity~$v$, and $T$-weights of $L^{\perp}$ is $-\chi_0$ with multiplicity $\dim L^{\perp}=c$. Thus, if $c=qv$, the sum of all $T$-weights of $Y$ is trivial, and in this case the category $\Db(\coh \cB_L)$ is Calabi--Yau of dimen\-sion~$2qs-1$.
\end{rem}
\appendix

\section{Algebras over groupoids in algebraic spaces}\label{appendix}

In this appendix, we prove a generalization of Proposition~\ref{quotient correspond}. We freely use the terminology and notation from~\cite{stacks}, and categories fibered in groupoids are defined over the big fppf site $(\Sch/S)_{\fppf}$ over a fixed base scheme $S$. First, we recall some definitions from~\cite{stacks}, where the main references are Chapters {\it Groupoids in Algebraic Spaces, Algebraic Stacks}, and {\it Sheaves on Algebraic Stacks}.

Let $p\colon \cX\to (\Sch/S)_{\fppf}$ be a category fibered in groupoids. Recall that we have the induced fppf topology on $\cX$, where a family of morphisms $\{x_i\to x\}_{i}$ in $\cX$ is a {\it covering} of $x\in \cX$ if the family $\{ p(x_i)\to p(x)\}_i$ is an fppf cover of the scheme $p(x)$. We write $\cX_{\fppf}$ for this fppf site. Then we have the {\it structure sheaf} $\cO_{\cX}\colon \cX_{\fppf}^{\op}\to (\Rings)$ of $\cX_{\fppf}$ defined by
\begin{gather*}
\cO_{\cX}(x)\defeq\Gamma\big(p(x),\cO_{p(x)}\big),
\end{gather*}
and thus we have the associated ringed site $\big(\cX_{\fppf},\cO_{\cX}\big)$. An $\cO_{\cX}$-module $\cF\colon\cX_{\fppf}^{\op}\to (\Sets)$ is said to be {\it quasi-coherent}, if for any $x\in \cX$ there is a covering $\{x_i\to x\}$ of $x$ such that for each~$x_i$ there is an exact sequence of $\cO_{\cX/x_i}$-modules
\begin{gather*}
\bigoplus_{I}\cO_{\cX/x_i}\to \bigoplus_{J}\cO_{\cX/x_i}\to\cF|_{x_i}\to 0,
\end{gather*}
 where $\big(\cX_{\fppf}/x_i,\cO_{\cX/x_i}\big)$ is the localization of the ringed site $\big(\cX_{\fppf},\cO_{\cX}\big)$ at $x_i\in\cX$ and $\cF|_{x_i}$ is the restriction of $\cF$ to $\big(\cX_{\fppf}/x_i,\cO_{\cX/x_i}\big)$. We denote by $\Qcoh \cX$ the category of quasi-coherent $\cO_{\cX}$-modules. If $\cX$ is an algebraic stack, $\Qcoh \cX$ is equivalent to the category of quasi-coherent sheaves on the lisse-\'etale site of $\cX$ defined in~\cite[Definition 9.1.14]{olsson}. In particular, if $\cX$ is an algebraic space, $\Qcoh \cX$ is equivalent to the category of quasi-coherent sheaves on the small \'etale site of $\cX$ defined in~\cite[Definition 7.1.5]{olsson}, and if $\cX$ is a scheme, it is equivalent to the category of usual quasi-coherent sheaves on the small Zariski site of $\cX$. If $X$ is an algebraic space, we write $X_{\et}$ (resp.\ $X_{\fppf}$) for the small \'etale site of $X$ (resp.\ the big fppf site $(\Sch/X)_{\fppf}$ of $X$), and sheaves on $X$ means sheaves on $X_{\et}$. Similarly, sheaves on a scheme $X$ means sheaves on the small Zariski site of $X$.

 \begin{dfn} An {\it ${\cX}$-algebra} is a sheaf of (not necessarily commutative) rings $\cA$ on $\cX_{\fppf}$ together with a morphism of sheaves of rings
 \begin{gather*}
 \uprho\colon\ \cO_{\cX}\to \cA
 \end{gather*}
 such that the $\cO_{\cX}$-module $\cA$ is quasi-coherent and the image of $\uprho$ is in the center of $\cA$, i.e., for any $x\in \cX$ the image of $\uprho(x)\colon \cO_{\cX}(x)\to \cA(x)$ is contained in the center of the ring $\cA(x)$.
\end{dfn}

 If $(\cA,\uprho)$ is an $\cX$-algebra, we have the restriction by $\uprho$
 \begin{gather*}
 (-)_{\uprho}\colon\quad \Mod\cA\to \Mod\cO_{\cX}, \qquad \cM\mapsto \cM_{\uprho},
 \end{gather*}
 where $ \Mod\cA$ denotes the category of right $\cA$-modules.
 \begin{dfn} A right $\cA$-module $\cM$ is said to be {\it quasi-coherent} if the $\cO_{\cX}$-module $\cM_{\uprho}$ is quasi-coherent. We denote by
 \begin{gather*}
 \Qcoh\cA
 \end{gather*}
 the full subcategory of $\Mod\cA$ consisting of quasi-coherent right $\cA$-modules.
 \end{dfn}

 \begin{rem}
 If $\cX$ is an algebraic space, we tacitly consider $\cX$-algebras and quasi-coherent modules over $\cX$-algebras as sheaves on the small \'etale site $\cX_{\et}$ instead of big fppf sheaves as above.
 \end{rem}

Let $\cA$ be an $\cX$-algebra, and denote by $\widetilde{\cX}$ the stackification of $\cX$. By~\cite[Lemma~12.1; 06WQ]{stacks}, we have an equivalence of topoi
\begin{equation}\label{equiv topoi}
\Sh(\cX_{\fppf})\simto \Sh\big(\widetilde{\cX}_{\fppf}\big),
\end{equation}
where $\Sh(-)$ denotes the category of sheaves on a site $(-)$. For an object $\cF\in \Sh(\cX_{\fppf})$, we denote by $\widetilde{\cF}\in\Sh\big(\widetilde{\cX}_{\fppf}\big)$ the image of $\cF$ by the equivalence \eqref{equiv topoi}.
By this equivalence the ring object $\cA\in\Sh(\cX_{\fppf})$ defines a ring object $\widetilde{\cA}\in\Sh\big(\widetilde{\cX}_{\fppf}\big)$ and a morphism
\begin{gather*}
\widetilde{\uprho}\colon\ \cO_{\widetilde{\cX}}\to \widetilde{\cA}.
\end{gather*}
By~\cite[Lemma~12.2; 06WR]{stacks}, the sheaf of rings $\widetilde{\cA}$ is quasi-coherent $\cO_{\widetilde{\cX}}$-module, and so $\widetilde{\cA}$ is an~$\widetilde{\cX}$-algebra.

\begin{lem}\label{stackification equiv}
The equivalence of topoi $\Sh(\cX_{\fppf})\simto \Sh\big(\widetilde{\cX}_{\fppf}\big)$ induces the equivalence
\begin{gather*}
\Mod\cA\simto \Mod\widetilde{\cA}
\end{gather*}
of right modules, and it restricts to the equivalence of quasi-coherent modules
\begin{gather*}
\Qcoh\cA\simto \Qcoh\widetilde{\cA}.
\end{gather*}
\begin{proof}
The first assertion is obvious, and for the second statement it is enough to show that a right $\cA$-module $\cM$ is quasi-coherent if and only if the right $\widetilde{\cA}$-module $\widetilde{\cM}$ is quasi-coherent. But this follows from the following commutative diagram
\[
\begin{tikzcd}
\Mod\cA\arrow[rr,"\widetilde{(-)}"]\arrow[d,"(-)_{\uprho}"']&&\Mod\widetilde{\cA}\arrow[d," (-)_{\widetilde{\uprho}}"]\\
\Mod\cO_{\cX}\arrow[rr, "\widetilde{(-)}"]&&\Mod\cO_{\widetilde{\cX}}
\end{tikzcd}
\]
and~\cite[Lemma~12.2; 06WR]{stacks}.
\end{proof}
\end{lem}

Let $\cG=(U,R,s,t, c)$ be a groupoid in algebraic spaces over $S$ (see~\cite[Definition 11.1; 043W]{stacks} for the notation).

\begin{dfn}
 An {\it algebra over $\cG$}, or {\it $\cG$-algebra}, is a $U$-algebra $\cA\colon U_{\et}^{\op}\to (\Rings)$ together with an isomorphism
\begin{gather*}
\uptheta^{\cA}\colon\ s^*\cA\simto t^*\cA
\end{gather*}
of $R$-algebras such that the equations
\begin{equation}\label{groupoid alg}
p_1^*\uptheta^{\cA}\circ p_2^*\uptheta^{\cA}=c^*\uptheta^{\cA}
\qquad\text{and}\qquad
e^*\uptheta^{\cA}=\id_{\cA}
\end{equation}
of morphisms of sheaves of rings hold, where $p_i\colon R\times_{s,U,t}R \to R$ is the $i$-th projection and $e\colon U\to R$ is the identity.
\end{dfn}

\begin{rem}\label{str equiv}
Note that the structure sheaf $\cO_U$ together with the canonical isomorphism
\begin{gather*}
\uptheta^{\can}\colon\ s^*\cO_U\simto t^*\cO_U
\end{gather*}
defines a $\cG$-algebra $(\cO_U, \uptheta^{\can})$. If $\big(\cA,\uptheta^{\cA}\big)$ is a $\cG$-algebra, since $\uptheta^{\cA}$ is $\cO_R$-linear, the diagram
\[
\begin{tikzcd}
s^*\cO_U\arrow[rr,"\sim"]\arrow[d,"s^*\uprho"']\arrow[rrrr, bend left=20,swap, "\uptheta^{\can}"']&&\cO_R\arrow[lld," "]\arrow[rrd," "]\arrow[rr,"\sim"]&&t^*\cO_U\arrow[d,"t^*\uprho"]\\
s^*\cA\arrow[rrrr, "\uptheta^{\cA}"]&&&&t^*\cA
\end{tikzcd}
\]
is commutative, where the isomorphisms on the top level are natural isomorphisms.
\end{rem}

\begin{dfn} Notation is the same as above.
\begin{itemize}\itemsep=0pt
\item[1.] A {\it quasi-coherent module} over a $\cG$-algebra $\big(\cA,\uptheta^{\cA}\big)$ is a quasi-coherent right $\cA$-module $\cM\colon U^{\op}_{\et}\to (\Sets)$ together with an isomorphism
\begin{gather*}
\uptheta^{\cM}\colon\ s^*\cM\simto t^*\cM
\end{gather*}
of quasi-coherent right $s^*\cA$-modules such that the similar equations as \eqref{groupoid alg} hold, where the right $t^*\cA$-module $t^*\cM$ is considered as a right $s^*\cA$-module via the isomorphism $\uptheta^{\cA}\colon s^*\cA\simto t^*\cA$ of $R$-algebras.
\item[2.] Let $\big(\cM,\uptheta^{\cM}\big)$ and $\big(\cN,\uptheta^{\cN}\big)$ be quasi-coherent modules over a $\cG$-algebra $\big(\cA,\uptheta^{\cA}\big)$.
A {\it morphism} from $\big(\cM,\uptheta^{\cM}\big)$ to $\big(\cN,\uptheta^{\cN}\big)$ is a morphism
\begin{gather*}
\varphi\colon\ \cM\to \cN
\end{gather*}
of right $\cA$-modules such that the diagram
\[
\begin{tikzcd}
s^*\cM\arrow[rr,"s^*\varphi"]\arrow[d,"\uptheta^{\cM}"']&&s^*\cN\arrow[d," \uptheta^{\cN}"]\\
t^*\cM\arrow[rr, "t^*\varphi"]&&t^*\cN
\end{tikzcd}
\]
is commutative.
\end{itemize}
\end{dfn}

We denote by
\begin{gather*}
\Qcoh\big(\cA,\uptheta^{\cA}\big)
\end{gather*}
the category of quasi-coherent modules over a $\cG$-algebra $\big(\cA,\uptheta^{\cA}\big)$.

\begin{rem}
We have a natural identification
\begin{gather*}
\Qcoh\big(\cO_U,\uptheta^{\can}\big)=\Qcoh(U,R,s,t,c),
\end{gather*}
where $\Qcoh(U,R,s,t,c)$ denotes the category of quasi-coherent modules on $(U,R,s,t,c)$ in the sense of~\cite[Definition 12.1; 0441]{stacks}.
\end{rem}

Let $\cG=(U,R,s,t, c)$ be a groupoid in algebraic spaces over $S$. Following~\cite[044O]{stacks} we denote by $[U/_pR]$ the category fibered in groupoids associated to $\cG$, which is defined as follows: An object of $[U/_pR]$ is a pair $(T,u)$ of $S$-scheme $T$ and a morphism $u\colon T\to U$ of algebraic spaces over $S$, and a morphism $(f,\varphi)\colon(T,u)\to (T',u')$ is defined to be a pair of a morphism $f\colon T\to T'$ of $S$-schemes and a morphism $\varphi\colon T\to R$ of algebraic spaces over $S$ such that $s\circ \varphi=u$ and $t\circ \varphi=u'\circ f$. Then the functor
\begin{gather*}
[U/_pR]\to (\Sch/S), \qquad (T,u)\mapsto T
\end{gather*}
is a category fibered in groupoids.
We denote by $[U/R]$ the stackification of $[U/_pR]$, which is called the {\it quotient stack} of $\cG$. We have the natural morphism
\begin{equation}\label{atlas}
\alpha\colon\ (\Sch/U)\to [U/_pR]
\end{equation}
defined by $\alpha(u\colon T\to U)\defeq(T,u)$ and $\alpha(f\colon T\to T')\defeq(f, e\circ u)$, and it defines the functor
\begin{gather*}
\alpha^{-1}\colon\ \Sh\big([U/_pR]_{\fppf}\big)\to \Sh(U_{\fppf})
\end{gather*}
defined by $\alpha^{-1}\cF(u\colon T\to U)\defeq \cF(\alpha(u\colon T\to U))=\cF(T,u)$.

For a $\cG$-algebra $\big(\cA,\uptheta^{\cA}\big)$, we define a sheaf of rings
\begin{gather*}
[\cA/_p R]\colon\ [U/_pR]_{\fppf}^{\op}\to (\Rings)
\end{gather*}
on $[U/_pR]_{\fppf}$ as follows:
For an object $(T,u)\in [U/_pR]$ we set
\begin{gather*}
[\cA/_p R](T,u)\defeq\Gamma(T,u^*\cA),
\end{gather*}
and for a morphism $(f,\varphi)\colon (T,u)\to (T',u')$ we define a morphism of rings
\begin{gather*}
[\cA/_p R](f,\varphi)\colon\ [\cA/_p R](T',u')\to [\cA/_p R](T,u)
\end{gather*}
by the following commutative diagram:
\[
\begin{tikzcd}
\Gamma(T',u'^*\cA)\arrow[rrr,"{[\cA/_p R](f,\varphi)}"]\arrow[d,"f^*"']&&&\Gamma(T,u^*\cA)\\
\Gamma(T,f^*u'^*\cA)\arrow[r, "\sim"]&\Gamma(T,\varphi^*t^*\cA)\arrow[rr,"\varphi^*{\uptheta^{\cA}}"]&&\Gamma(T,\varphi^*s^*\cA)\arrow[u, "\sim"' labl].
\end{tikzcd}
\]

 Similarly, for a quasi-coherent module $\big(\cM,\uptheta^{\cM}\big)$ over a $\cG$-algebra $\big(\cA,\uptheta^{\cA}\big)$, we define the right $[\cA/_pR]$-module
\begin{gather*}
[\cM/_pR]\colon\ [U/_pR]_{\fppf}^{\op}\to (\Sets)
\end{gather*}
by $[\cM/_p R](T,u)\defeq\Gamma(T,u^*\cM)$, and this defines an additive functor
\begin{equation}\label{quot functor}
[(-)/_pR]\colon\ \Qcoh\big(\cA,\uptheta^{\cA}\big)\to \Mod[\cA/_pR].
\end{equation}
By the equality $[\cO_U/_pR]=\cO_{[U/_pR]}$ and Remark~\ref{str equiv}, the morphism $\uprho\colon \cO_U\to \cA$ defines a morphism from $(\cO_U,\uptheta^{\can})$ to $\big(\cA,\uptheta^{\cA}\big)$ in $\Qcoh\big(\cA,\uptheta^{\cA}\big)$, and thus we have the induced morphism
\begin{gather*}
[\uprho/_pR]\colon\ \cO_{[U/_pR]}\to [\cA/_pR]
\end{gather*}
of sheaves of rings.

\begin{lem}\label{qcoh qcoh} Notation is the same as above.
\begin{itemize}\itemsep=0pt
\item[$1.$] If $(\cM,\uptheta^{\can})\in \Qcoh(\cO_U,\uptheta^{\can})$, then $[\cM/_pR]\in \Qcoh[U/_pR]$.
\item[$2.$] The sheaf $[\cA/_pR]$ together with the ring homomorphism $[\uprho/_pR]$ is a $[U/_pR]$-algebra.
\item[$3.$] If $\big(\cM,\uptheta^{\cM}\big)\in \Qcoh\big(\cA,\uptheta^{\cA}\big)$, the right $[\cA/_pR]$-module $[\cM/_pR]$ is quasi-coherent.
\end{itemize}
\begin{proof}
(1) Let $\{a_i\colon A_i\to U\}_{i\in I}$ be an affine \'etale covering of $U$. Since $\cM$ is quasi-coherent $\cO_U$-module, for each $i\in I$ there is an exact sequence
\begin{equation}\label{qcoh exact}
\bigoplus_{J}\cO_{A_i}\to \bigoplus_{K}\cO_{A_i}\to a_i^*\cM\to 0.
\end{equation}
For any object $(T,u)\in [U/_pR]$, set $T_i\defeq A_i\times_{U}T$ and write $t_i\colon T_i\to T$ for the second projection. If we set $u_i\defeq u\circ t_i\colon T_i\to U$ and $\varphi_i\defeq e\circ u_i\colon T_i\to R$, then
\begin{gather*}
\bigl\{(t_i,\varphi_i)\colon (T_i, u_i)\to (T,u)\bigr\}_{i\in I}
\end{gather*}
is a covering of $(T,u)\in [U/_pR]_{\fppf}$. Since the localized ringed site
\begin{gather*}
\big([U/_pR]_{\fppf}/(T_i,u_i),\cO_{[U/_pR]_{\fppf}/(T_i,u_i)}\big)
\end{gather*}
is equivalent to the ringed site $\big(({\Sch/T_i})_{\fppf},\cO_{({\Sch/T_i})_{\fppf}}\big)$ by~\cite[Lemma~9.4; 06W9]{stacks}, the pull-back of the exact sequence \eqref{qcoh exact} by the first projection $T_i\to A_i$ gives rise to an exact sequence
\begin{gather*}
\bigoplus_{J}\cO_{[U/_pR]_{\fppf}/(T_i,u_i)}\to \bigoplus_{K}\cO_{[U/_pR]_{\fppf}/(T_i,u_i)}\to \cM|_{(T_i,u_i)}\to 0.
\end{gather*}

(2) Since the image of $\uprho\colon \cO_U\to \cA$ is contained in the center of $\cA$, the image of \begin{gather*}[\uprho/_pR]\colon \cO_{[U/_pR]}\to [\cA/_pR]\end{gather*} is contained in the center of $[\cA/_pR]$. The $\cO_{[U/_pR]}$-module $[\cA/_pR]$ is quasi-coherent by (1).

(3) This follows from (1).
\end{proof}
\end{lem}

By Lemma~\ref{qcoh qcoh}, the functor \eqref{quot functor} defines an additive functor
\begin{gather*}
[(-)/_pR]\colon\ \Qcoh\big(\cA,\uptheta^{\cA}\big)\to \Qcoh[\cA/_pR].
\end{gather*}
\begin{prop}\label{quot equiv}
The functor $[(-)/_pR]\colon \Qcoh\big(\cA,\uptheta^{\cA}\big)\to \Qcoh[\cA/_pR]$ is an equivalence.
\begin{proof}
This follows from a similar argument as in the proof of~\cite[Proposition~13.1; 06WT]{stacks}. We construct a quasi-inverse
of the functor $[(-)/_pR]$ as follows.

First, for a sheaf $\cF$ on $[U/_pR]_{\fppf}$, we define a sheaf $\cF^R$ on the small \'etale site $U_{\et}$ by
\begin{gather*}
\cF^R\defeq\big(\alpha^{-1}\cF\big)|_{U_{\et}},
\end{gather*}
where $\alpha\colon (\Sch/U)\to [U/_pR]$ is the morphism in \eqref{atlas}. By construction we have $\cO_{[U/_pR]}^R=\cO_{U_{\et}}$. For each \'etale morphism $(u\colon T\to U)\in U_{\et}$, we have
\begin{gather*}
[\cA/_pR]^R(u\colon T\to U)=\Gamma(T,u^*\cA)=\cA(u\colon T\to U).
\end{gather*}
 Thus we have a natural identification $[\cA/_pR]^R=\cA$, and so if $\cM$ is a quasi-coherent right $[\cA/_pR]$-module, the small \'etale sheaf $\cM^R$ is a quasi-coherent $\cA$-module.

Next, we define a category $\{ U/_pR\}$ containing the category $[U/_pR]$ by replacing $S$-sche\-mes~$T$ in the definition of $[U/_pR]$ with algebraic spaces over $S$. If $(f,\varphi)\colon (W,u)\to(W',u')$ is a~morphism in $\{U/_pR\}$, this morphism induces a natural (functorial) commutative diagram
\[
\begin{tikzcd}
W \arrow[rr,"f"]\arrow[rd,"\alpha\circ u"']&&W'\arrow[ld," \alpha\circ u'"]\\
&\left[U/_pR\right].&
\end{tikzcd}
\]
Then, for any quasi-coherent right $[\cA/_pR]$-module $\cM$, we have the composition
\begin{gather*}
(\alpha\circ u')^{-1}\cM\to f_*f^{-1}\big((\alpha\circ u')^{-1}\cM\big)\simto f_*(\alpha\circ u)^{-1}\cM,
\end{gather*}
where the first morphism is the adjunction morphism and the second isomorphism follows from the above commutative diagram. By the adjunction $f^{-1}\dashv f_*$, this composition map induces the morphism
$f^{-1}(\alpha\circ u')^{-1}\cM\to (\alpha\circ u)^{-1}\cM$
and restricting this morphism to $W_{\et}$, we have the comparison map
\begin{gather*}
c_{(f,\varphi)}\colon\ f^*u'^*\cM^R\to u^*\cM^R
\end{gather*}
and by~\cite[Lemma~11.6; 06WK]{stacks} this is an isomorphism.

Finally, we define a functor $F\colon \Qcoh[\cA/_pR]\to \Qcoh\big(\cA,\uptheta^{\cA}\big)$ by
\begin{gather*}
F(\cM)\defeq\big(\cM^R, \uptheta^{\cM^R}\big),
\end{gather*}
where $\uptheta^{\cM^R}\colon s^*\cM^R\simto t^*\cM$ is defined to be the inverse of the comparison map
\begin{gather*}
c_{(\id_R,\id_R)}\colon\ t^*\cM^R\simto s^*\cM^R
\end{gather*}
induced by the morphism $(\id_R,\id_R)\colon (R,s)\to (R,t)$ in $\{U/_pR\}$. By construction, $F$ is a quasi-inverse of the functor $[(-)/_pR]$.
\end{proof}
\end{prop}

Write
\begin{gather*}
[\cA/R]\defeq\widetilde{[\cA/_pR]}
\end{gather*}
for the sheaf of rings on the quotient stack $[U/R]$ corresponding to $[\cA/_pR]$ via the equivalence~\eqref{equiv topoi}. By Lemma~\ref{stackification equiv} and Proposition~\ref{quot equiv}, we have the following
 generalization of~\cite[Proposition~13.1; 06WT]{stacks}.

\begin{prop}\label{main app}
We have an equivalence
\begin{gather*}
\Qcoh\big(\cA,\uptheta^{\cA}\big)\simto \Qcoh[\cA/R].
\end{gather*}
\end{prop}

\subsection*{Acknowledgements}

The author is supported by JSPS KAKENHI 19K14502. Part of this work was done during his stay at the Max Planck Institute for Mathematics in Bonn, in the period from April to~Decem\-ber 2019. The author gratefully acknowledge MPIM Bonn for their support and hospitality. Furthermore, the author would like to thank the referees for their careful reading and valuable suggestions.


\pdfbookmark[1]{References}{ref}
\LastPageEnding


\begin{thebibliography}{99}
\footnotesize\itemsep=0pt

\bibitem{tlrv}
Alonso~Tarr\'{\i}o L., Jerem\'{\i}as~L\'opez A., P\'erez~Rodr\'{\i}guez M.,
 Vale~Gonsalves M.J., A functorial formalism for quasi-coherent sheaves on a
 geometric stack, \href{https://doi.org/10.1016/j.exmath.2014.12.007}{\textit{Expo. Math.}} \textbf{33} (2015), 452--501,
 \href{https://arxiv.org/abs/1304.2520}{arXiv:1304.2520}.


\bibitem{bdfik}
Ballard M., Deliu D., Favero D., Isik M.U., Katzarkov L., Resolutions in
 factorization categories, \href{https://doi.org/10.1016/j.aim.2016.02.008}{\textit{Adv. Math.}} \textbf{295} (2016), 195--249,
 \href{https://arxiv.org/abs/1212.3264}{arXiv:1212.3264}.

\bibitem{bfk}
Ballard M., Favero D., Katzarkov L., A category of kernels for equivariant
 factorizations and its implications for {H}odge theory, \href{https://doi.org/10.1007/s10240-013-0059-9}{\textit{Publ. Math.
 Inst. Hautes \'Etudes Sci.}} \textbf{120} (2014), 1--111,
 \href{https://arxiv.org/abs/1105.3177}{arXiv:1105.3177}.

\bibitem{bfk2}
Ballard M., Favero D., Katzarkov L., A category of kernels for equivariant
 factorizations, {II}: further implications, \href{https://doi.org/10.1016/j.matpur.2014.02.004}{\textit{J.~Math. Pures Appl.}}
 \textbf{102} (2014), 702--757, \href{https://arxiv.org/abs/1310.2656}{arXiv:1310.2656}.

\bibitem{hl-s}
Halpern-Leistner D., Sam S.V., Combinatorial constructions of derived
 equivalences, \href{https://doi.org/10.1090/jams/940}{\textit{J.~Amer. Math. Soc.}} \textbf{33} (2020), 735--773,
 \href{https://arxiv.org/abs/1601.02030}{arXiv:1601.02030}.

\bibitem{happel}
Happel D., Triangulated categories in the representation theory of
 finite-dimensional algebras, \textit{London Ma\-the\-matical Society Lecture Note
 Series}, Vol.~119, \href{https://doi.org/10.1017/CBO9780511629228}{Cambridge University Press}, Cambridge, 1988.

\bibitem{H2}
Hirano Y., Derived {K}n\"orrer periodicity and {O}rlov's theorem for gauged
 {L}andau--{G}inzburg models, \href{https://doi.org/10.1112/S0010437X16008344}{\textit{Compos. Math.}} \textbf{153} (2017),
 973--1007, \href{https://arxiv.org/abs/1602.04769}{arXiv:1602.04769}.

\bibitem{H1}
Hirano Y., Equivalences of derived factorization categories of gauged
 {L}andau--{G}inzburg models, \href{https://doi.org/10.1016/j.aim.2016.10.023}{\textit{Adv. Math.}} \textbf{306} (2017),
 200--278, \href{https://arxiv.org/abs/1506.00177}{arXiv:1506.00177}.

\bibitem{ks}
Kashiwara M., Schapira P., Sheaves on manifolds, \textit{Grundlehren der
 Mathematischen Wissenschaften}, Vol.~292, \href{https://doi.org/10.1007/978-3-662-02661-8}{Springer-Verlag}, Berlin, 1994.

\bibitem{ks2}
Kashiwara M., Schapira P., Categories and sheaves, \textit{Grundlehren der
 Mathematischen Wissenschaften}, Vol.~332, \href{https://doi.org/10.1007/3-540-27950-4}{Springer-Verlag}, Berlin, 2006.

\bibitem{lip}
Lipman J., Notes on derived functors and {G}rothendieck duality, in Foundations
 of {G}rothendieck duality for diagrams of schemes, \textit{Lecture Notes in
 Math.}, Vol.~1960, \href{https://doi.org/10.1007/978-3-540-85420-3}{Springer}, Berlin, 2009, 1--259.

\bibitem{ls}
Lunts V.A., Schn\"urer O.M., Matrix factorizations and semi-orthogonal
 decompositions for blowing-ups, \href{https://doi.org/10.4171/JNCG/252}{\textit{J.~Noncommut. Geom.}} \textbf{10}
 (2016), 907--979, \href{https://arxiv.org/abs/1212.2670}{arXiv:1212.2670}.

\bibitem{ot}
Okonek C., Teleman A., Graded tilting for gauged {L}andau--{G}inzburg models
 and geometric applications, \href{https://dx.doi.org/10.4310/PAMQ.2021.v17.n1.a5}{\textit{Pure Appl. Math.~Q.}} \textbf{17} (2021),
 185--235, \href{https://arxiv.org/abs/1907.10099}{arXiv:1907.10099}.

\bibitem{olsson}
Olsson M., Algebraic spaces and stacks, \textit{American Mathematical Society
 Colloquium Publications}, Vol.~62, \href{https://doi.org/10.1090/coll/062}{Amer. Math. Soc.}, Providence, RI, 2016.

\bibitem{posi}
Positselski L., Two kinds of derived categories, {K}oszul duality, and
 comodule-contramodule correspondence, \href{https://doi.org/10.1090/S0065-9266-2010-00631-8}{\textit{Mem. Amer. Math. Soc.}}
 \textbf{212} (2011), vi+133~pages, \href{https://arxiv.org/abs/0905.2621}{arXiv:0905.2621}.

\bibitem{rs}
Rennemo J.V., Segal E., Hori-mological projective duality, \href{https://doi.org/10.1215/00127094-2019-0014}{\textit{Duke
 Math.~J.}} \textbf{168} (2019), 2127--2205, \href{https://arxiv.org/abs/1609.04045}{arXiv:1609.04045}.

\bibitem{rsv}
Rennemo J.V., Segal E., Van~den Bergh M., A non-commutative {B}ertini theorem,
 \href{https://doi.org/10.4171/JNCG/334}{\textit{J.~Noncommut. Geom.}} \textbf{13} (2019), 609--616,
 \href{https://arxiv.org/abs/1705.01366}{arXiv:1705.01366}.

\bibitem{svdb}
\v{S}penko \v{S}., Van~den Bergh M., Non-commutative resolutions of quotient
 singularities for reductive groups, \href{https://doi.org/10.1007/s00222-017-0723-7}{\textit{Invent. Math.}} \textbf{210}
 (2017), 3--67, \href{https://arxiv.org/abs/1502.05240}{arXiv:1502.05240}.

\bibitem{stacks}
The Stacks Project Authors, Stacks Project, available at \url{https://stacks.math.columbia.edu}.

\bibitem{tho}
Thomason R.W., Equivariant resolution, linearization, and {H}ilbert's
 fourteenth problem over arbitrary base schemes, \href{https://doi.org/10.1016/0001-8708(87)90016-8}{\textit{Adv. Math.}}
 \textbf{65} (1987), 16--34.

\end{thebibliography}
\end{document}